\documentclass[11pt,letterpaper,reqno]{amsart}

\title[Geometric Control of Spherical Rolling Robot]{Geometric Kinematic Control of a Spherical Rolling Robot}
\author{Tomoki Ohsawa}
\address{Department of Mathematical Sciences, The University of Texas at Dallas, 800 W Campbell Rd, Richardson, TX 75080-3021}
\email{tomoki@utdallas.edu}
\date{\today}



\usepackage{graphicx,mathrsfs,paralist}
\usepackage[margin=1in, marginpar=.5in]{geometry}

\usepackage{caption,subfigure}

\usepackage{tikz}

\usepackage[numbers,sort&compress]{natbib}

\usepackage[colorlinks=false]{hyperref}

\theoremstyle{plain}
\newtheorem{theorem}{Theorem}[section]

\theoremstyle{definition}

\newtheorem{example}[theorem]{Example}

\theoremstyle{remark}
\newtheorem{remark}[theorem]{Remark}

\def\od#1#2{\dfrac{d#1}{d#2}}
\def\pd#1#2{\dfrac{\partial #1}{\partial #2}}

\def\parentheses#1{{\left(#1\right)}}
\def\brackets#1{{\left[#1\right]}}

\def\Span{\mathop{\mathrm{span}}\nolimits} 
\def\argmax{\mathop{\mathrm{argmax}}}
\def\sn{\mathop{\mathrm{sn}}\nolimits}

\def\norm#1{{\left\|#1\right\|}}

\def\DS{\displaystyle}
\def\R{\mathbb{R}}

\def\defeq{\mathrel{\mathop:}=}
\def\eqdef{=\mathrel{\mathop:}}
\def\setdef#1#2{{\left\{ #1 \ |\ #2 \right\}}}

\def\eps{\varepsilon}

\def\SO{\mathsf{SO}}

\def\so{\mathfrak{so}}

\def\d{\mathbf{d}}

\newcommand\Ad{\operatorname{Ad}}

\newcommand\hor{\operatorname{hor}}

\newcommand\hl{\operatorname{hl}}

\def\PB#1#2{\left\{#1,#2\right\}}

\def\bomega{\boldsymbol{\omega}}

\begin{document}

\footskip=.6in

\begin{abstract}
  We give a geometric account of kinematic control of a spherical rolling robot controlled by two internal wheels just like the toy robot \textit{Sphero}.
  Particularly, we introduce the notion of shape space and fibers to the system by exploiting its symmetry and the principal bundle structure of its configuration space; the shape space encodes the rotational angles of the wheels, whereas each fiber encodes the translational and rotational configurations of the robot for a particular shape.
  We show that the system is fiber controllable---meaning any translational and rotational configuration modulo shapes is reachable---as well as find exact expressions of the geometric phase or holonomy under some particular controls.
  We also solve an optimal control problem of the spherical robot, show that it is completely integrable, and find an explicit solution of the problem.
\end{abstract}

\maketitle

\section{Introduction}
\subsection{Spherical Rolling Robot}
A spherical rolling robot is a simple robot that has been studied extensively in many different forms from both theoretical and experimental points of view; see, e.g., \citet{BhAg2000} for several different types of realizations.
One of the realizations is the \textit{Sphericle} developed by \citet{BiBaPrGo1997}; it is a spherical rolling robot controlled by two internal wheels inside the spherical shell of the robot.
The \textit{Sphero}\textsuperscript{\textregistered} (see Fig.~\ref{fig:Sphero}) is a commercial realization of the \textit{Sphericle}; it is controlled by two internal wheels (white wheels near the bottom) actuated by motors in the electromechanical unit inside the sphere (the blue wheels above are idler wheels to sustain the unit).
\begin{figure}[hbtp]
  \centering
  \includegraphics[width=0.325\linewidth]{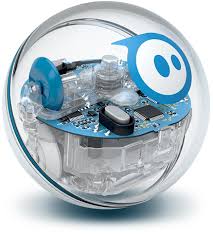}
  \caption{Spherical rolling robot \textit{Sphero}\textsuperscript{\textregistered}; see \url{http://www.sphero.com}.}
  \label{fig:Sphero}
\end{figure}

Despite its relative simplicity in design and configurations, a spherical rolling robot like the \textit{Sphero} has fairly complex motions due to its nonholonomic nature of the constraints.
In fact, the \textit{Sphero} comes with an interface that enables one to control it by using a cellphone app as well as both visual and traditional programming languages, and hence is an effective STEM education toy that teaches students basic ideas in computer programming, mechanics, and control theory.

\subsection{Main Results and Outline}
We study the kinematics of the spherical rolling robot like the \textit{Sphericle} or \textit{Sphero} from the geometric point of view.
Particularly, we exploit the symmetry of the kinematic model of the robot and the notion of \textit{shape space} (see, e.g., \citet{Mo1991a,Mo1993a,Mo1993b} and \citet{KeMu1995}), and analyze its controllability as well as its optimal control problem.

We first formulate the kinematic equation describing the nonholonomic constraints of the system in Section~\ref{sec:kinematics}.
The resulting model is effectively the same as that of the \textit{Sphericle} in \citet{BiBaPrGo1997}.

The main difference from their approach is that we stress the role of symmetry and formulate the system on a principal bundle; see Section~\ref{sec:geometry}.
In other words, we split the configuration space into the \textit{shape space} (configurations of the wheels or the internal system) and the \textit{fiber} (the symmetry group or the translational and rotational configurations of the sphere).
The system has a fully-actuated subsystem in the shape space, but the rest of the system in the direction of the fiber is not directly actuated and is defined by the constraint of the system.
However one is mainly concerned with the behavior of the system in the fiber.
This is the basic geometric setting for the Falling Cat Problem~\cite{Mo1993a} as well as robotic locomotion~\cite{KeMu1995}.

This leads to the question of fiber controllability~\cite{KeMu1995} of the robot in Section~\ref{sec:fiber_controllability}, i.e., whether the system is controllable in the fiber regardless of its shape.
We show that the system is fiber controllable by finding the curvature of its principal connection (Theorem~\ref{thm:fiber_controllability}).

We also demonstrate a couple of instances of \textit{holonomy} or \textit{geometric phase} in Section~\ref{sec:curvature_and_locomotion}.
A holonomy or geometric phase is the displacement in the fiber when the control system makes a loop in the shape space, i.e., when the shape of the system undergoes a change and eventually comes back to the original one.
We find exact expressions for translational and rotational holonomies under certain control laws that may be useful for motion planning.

Finally, in Section~\ref{sec:optimal_control}, we formulate an optimal control problem of the robot, and show that the system resulting from the Pontryagin Maximum Principle is completely integrable, as well as obtain an explicit solution to the problem (Theorem~\ref{thm:integrability}).

\section{Robot Kinematics}
\label{sec:kinematics}
\subsection{Simple Kinematic Model of \textit{Sphero}}
\label{ssec:kinematic_model}
We model the rolling robot under the following simplifying assumptions:
\begin{enumerate}[(i)]
\item The model is kinematic. (See, e.g., \cite{BiBaPrGo1997,Sc2002,ShScBl2008,PuRo2017,IlMoVl2017} for \textit{dynamical} studies of rolling robots.)
\item The electromechanical unit inside the robot always maintains its horizontal position.
\item There is no slip between the sphere and the ground in the sense that the contact points of both surfaces have the same velocity.
  \label{asmptn:no-slip1}
\item There is no slip between the sphere and the internal wheels in the same sense.
\end{enumerate}
It results in an essentially the same model as the \textit{Sphericle} developed by \citet{BiBaPrGo1997}.

Regarding the second assumption, the electromechanical unit has a ballast weight at the bottom and is much heavier than the spherical shell.
When the robot is in fast motion, the unit tilts and wobbles inside the sphere due to inertia.
However, when it is in slow motion, the unit stays at the bottom more or less maintaining its horizontal position at the bottom of the sphere due to its heavy weight relative to the spherical shell.
So we would like to model the kinematics of the robot assuming that the electromechanical unit can only rotate about the vertical axis.
In other words, we think of the robot as a spherical robot maneuvered by a two-wheeled unit rotating inside the spherical shell maintaining its horizontal position.

Regarding the third and fourth assumptions, these condition impose nonholonomic (rolling) constraints on the robot that define the kinematic system to consider; see Section~\ref{ssec:no-slip_constraints}.
Note that the third condition does not prevent the sphere from rotating about the vertical axis; see Section~\ref{ssec:additional_constraint} below.

\subsection{Kinematics of Rolling Sphere}
Let us first consider the kinematics of the sphere itself.
Consider the motion of a sphere with radius $r$ rolling on the plane $x_{3} = 0$ in the spatial frame $\R^{3} = \{(x_{1},x_{2},x_{3})\}$.
Following \citet{Ju1993} (see also \cite[Section~4.1]{Ju1997}), we describe the kinematics of the sphere as follows:
Let $\mathcal{S} \defeq \setdef{ \bar{\mathbf{q}} \in \R^{3}}{ \norm{\bar{\mathbf{q}}} = r }$ be the sphere in the body frame of the sphere.
The configuration of the rolling sphere is specified by the position of the center $\mathbf{x}_{\rm c} = (\mathbf{x},r) \in \R^{3}$ in the spatial frame with $\mathbf{x} = (x_{1},x_{2}) \in \R^{2}$ as well as the rotation matrix $R \in \SO(3)$ that specifies the orientation of the sphere in the spatial frame.
Hence the configuration space of the rolling sphere is $\SO(3) \times \R^{2} = \{(R,\mathbf{x})\}$.

Let $\bar{\mathbf{q}} \in \mathcal{S}$ be an arbitrary point on the sphere in the \textit{body} frame; see the sphere on the left in Fig.~\ref{fig:RollingSphere}.
For any given configuration $(R, \mathbf{x}) \in \SO(3) \times \R^{2}$ of the sphere, the position of the point of the sphere in the \textit{spatial} frame would be
\begin{equation*}
  \mathbf{q} \defeq \mathbf{x}_{\rm c} + R\bar{\mathbf{q}}
\end{equation*}
as shown on the right in Fig.~\ref{fig:RollingSphere}.
\begin{figure}[hbtp]
  \centering
  \includegraphics[width=0.75\linewidth]{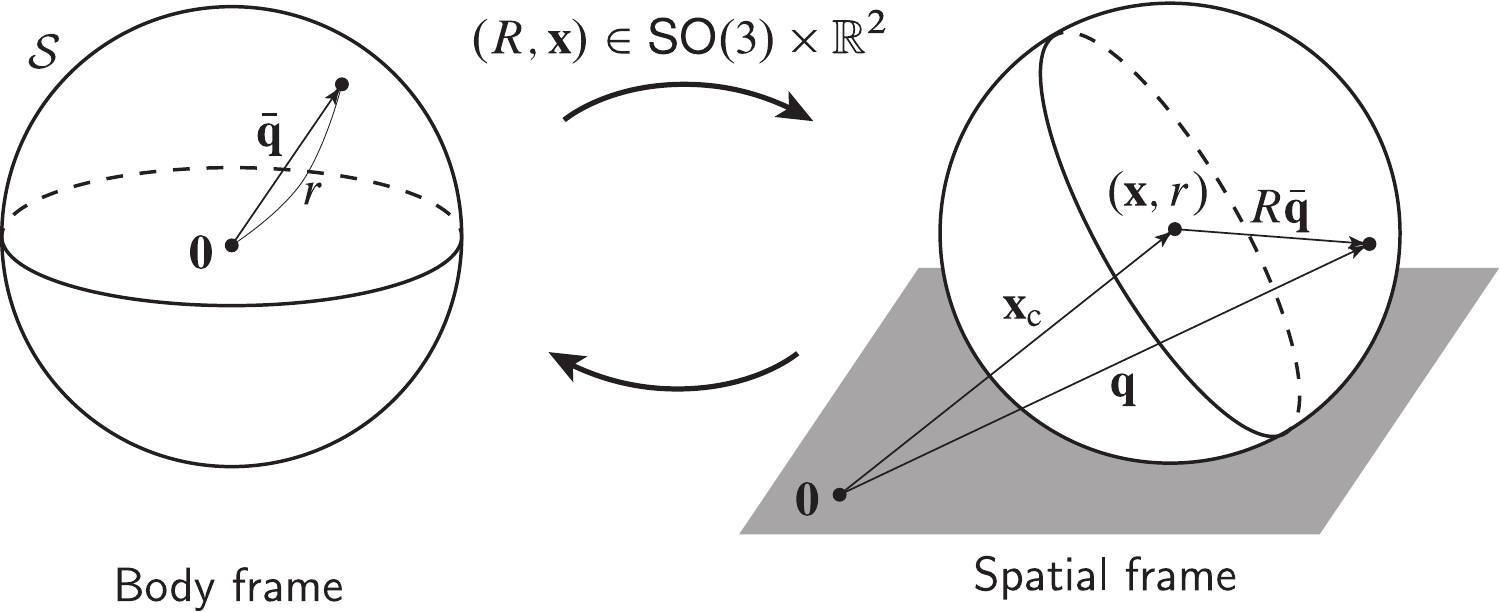}
  \caption{Rolling Sphere in the body and spatial frames}
  \label{fig:RollingSphere}
\end{figure}
We may write the velocity of the point $\mathbf{q}$ (in the spatial frame) in terms of $(\dot{R},\dot{\mathbf{x}}) \in T_{(R,\mathbf{x})}(\SO(3) \times \R^{2})$ as
\begin{equation}
  \label{eq:dotq}
  \dot{\mathbf{q}} 
  =
  \begin{bmatrix}
    \dot{\mathbf{x}} \\
    0
  \end{bmatrix}
  + \dot{R}\bar{\mathbf{q}}
  = \begin{bmatrix}
    \dot{\mathbf{x}} \\
    0
  \end{bmatrix}
  + \hat{\bomega} R\bar{\mathbf{q}}
  \eqdef f(R, \hat{\bomega}, \dot{\mathbf{x}}; \bar{\mathbf{q}}),
\end{equation}
where $\hat{\bomega}$ is the angular velocity in the spatial frame, i.e.,
\begin{equation*}
  \hat{\bomega} =
  \begin{bmatrix}
    0 & -\omega_{3} & \omega_{2} \\
    \omega_{3} & 0 & -\omega_{1} \\
    -\omega_{2} & \omega_{1} & 0 
  \end{bmatrix}
  \defeq \dot{R}R^{T} \in \so(3).
\end{equation*}
This is an example of the so-called ``hat map'' $\hat{(\,\cdot\,)}\colon \R^{3} \to \so(3)$ (see, e.g., \citet[Eq.~(9.2.7) on p.~289]{MaRa1999}) defined by
\begin{equation}
  \label{eq:hat_map}
  \mathbf{a} =
  \begin{bmatrix}
    a_{1} \\
    a_{2} \\
    a_{3}
  \end{bmatrix}
  \mapsto
  \hat{\mathbf{a}} =
  \begin{bmatrix}
    0 & -a_{3} & a_{2} \\
    a_{3} & 0 & -a_{1} \\
    -a_{2} & a_{1} & 0
  \end{bmatrix}.
\end{equation}

\subsection{No-slip Constraints}
\label{ssec:no-slip_constraints}
The no-slip condition of the contact point of the sphere with the plane imposes a nonholonomic constraint as follows:
For any given configuration $(R,\mathbf{x}) \in \SO(3) \times \R^{2}$ of the sphere, the contact point $\bar{\mathbf{q}}_{\rm c}(R)$ in the body frame satisfies $R\bar{\mathbf{q}}_{\rm c}(R) = -r\mathbf{e}_{3}$, i.e., $\bar{\mathbf{q}}_{\rm c}(R) = -r R^{T}\mathbf{e}_{3}$.
Hence, in view of \eqref{eq:dotq}, the velocity of the contact point in the spatial frame is:
\begin{equation*}
  f(R, \hat{\bomega}, \dot{\mathbf{x}}; \bar{\mathbf{q}}_{\rm c}(R))
  = f\parentheses{ R, \hat{\bomega}, \dot{\mathbf{x}}; -r R^{T}\mathbf{e}_{3} }
  = \begin{bmatrix}
    \dot{\mathbf{x}} \\
    0
  \end{bmatrix}
  - r\hat{\bomega}\mathbf{e}_{3}.
\end{equation*}
The no-slip condition says this vanishes, i.e.,
\begin{equation}
  \label{eq:no-slip}
  \dot{\mathbf{x}} =
  \begin{bmatrix}
    \dot{x}_{1} \\
    \dot{x}_{2}
  \end{bmatrix}
  = r
  \begin{bmatrix}
    \omega_{2} \\
    -\omega_{1}
  \end{bmatrix}.
\end{equation}

Let us now consider the kinematics of the robot, particularly the interaction between the sphere and the internal wheels.
Let $\psi \in \mathbb{S}^{1}$ be the angle of rotation of the two-wheeled unit measured from the positive part of the $x_{1}$-axis (of the spatial frame); see the left of Fig.~\ref{fig:Configuration}.
Let $\mathbf{y}_{\rm w}^{(i)} \in \R^{3}$ with $i = 1, 2$ be the position---relative to the center of the sphere in the body frame---of the contact point of wheel $i$ to the sphere in the body frame as shown on the right in Fig.~\ref{fig:Configuration}; the wheels are numbered as shown in the figure.
\begin{figure}[hbtp]
  \centering
  \includegraphics[width=0.75\linewidth]{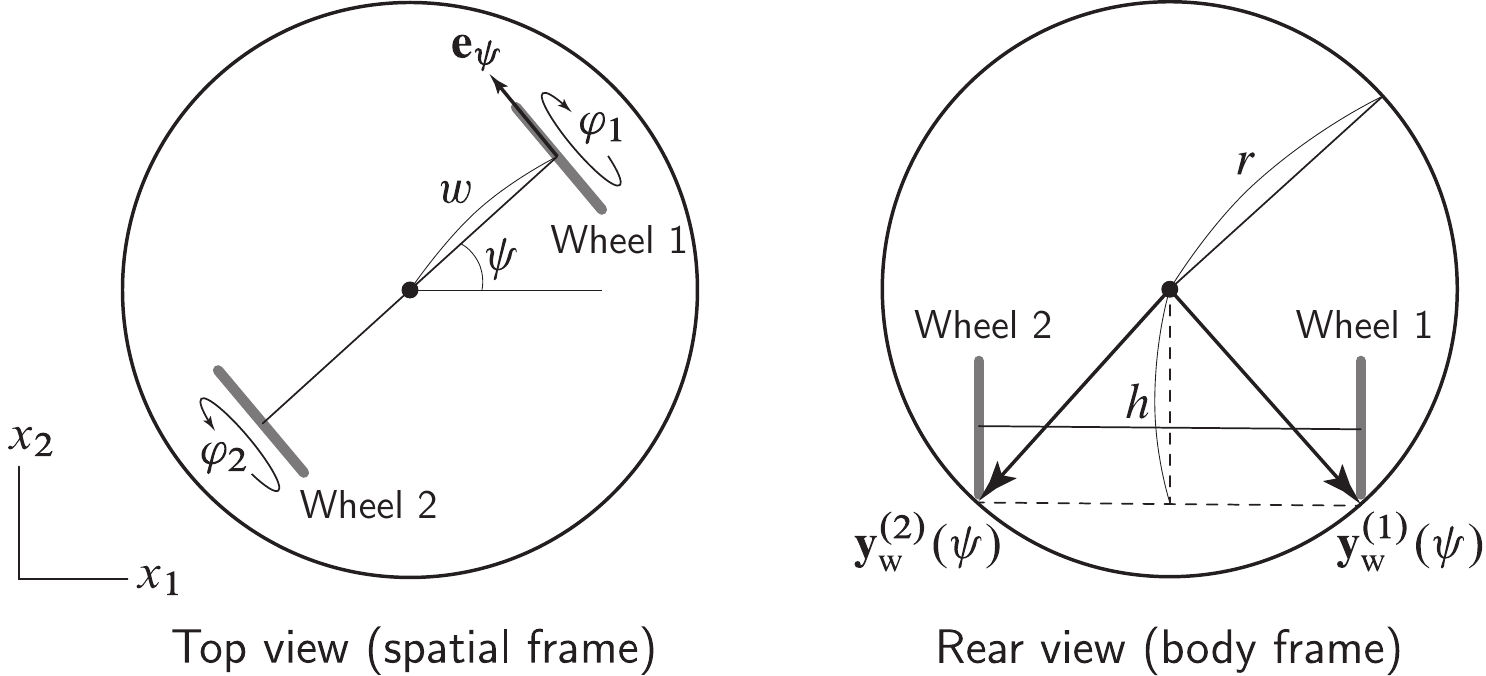}
  \caption{Configuration of Robot}
  \label{fig:Configuration}
\end{figure}
Let $2w$ be the track width, $h$ be the distance between the center of the sphere and the horizontal plane defined by the contact points of the wheels.
Then it is easy to see that
\begin{equation*}
  \mathbf{y}_{\rm w}^{(1)}(\psi) \defeq
  \begin{bmatrix}
    w \cos\psi \\
    w \sin\psi \\
    -h
  \end{bmatrix},
  \qquad
  \mathbf{y}_{\rm w}^{(2)}(\psi) \defeq
  \begin{bmatrix}
    -w \cos\psi \\
    -w \sin\psi \\
    -h
  \end{bmatrix}.
\end{equation*}

In addition to the no-slip condition of the sphere itself described above, the above model of the rolling robot imposes additional no-slip constraints at the contact points of the wheels to the sphere.
The constraints are simply that the velocity of the contact point of each wheel must match that of the sphere.

Let us first find the velocities of the contact points of the wheels.
Let $\rho$ be the radius of the wheels.
Then the positions of the contact points of the wheels in the spatial frame are
\begin{equation*}
  \mathbf{q}_{\mathrm{w}}^{(i)} \defeq \mathbf{x}_{c} + \mathbf{y}_{\rm w}^{(i)}(\psi)
  \quad
  \text{for}
  \quad
  i = 1, 2.
\end{equation*}
Then the velocity (in the spatial frame) of each wheel is the composition of the translational and rotational velocities of the electromechanical unit and the rotational velocity of the wheel itself; see Fig.~\ref{fig:Wheels}.
Hence the velocities of the contact points of the wheels in the spatial frame are given by
\begin{equation*}
  \dot{\mathbf{q}}_{\rm w}^{(1)} \defeq \dot{\mathbf{x}}_{\rm c} + (w\dot{\psi} - \rho\dot{\varphi}_{1})\mathbf{e}_{\psi},
  \qquad
  \dot{\mathbf{q}}_{\rm w}^{(2)} \defeq \dot{\mathbf{x}}_{\rm c} - (w\dot{\psi} + \rho\dot{\varphi}_{2})\mathbf{e}_{\psi},
\end{equation*}
where $\mathbf{e}_{\psi} \defeq (-\sin\psi, \cos\psi, 0)^{T}$ and is shown on the left in Fig.~\ref{fig:Configuration}.
\begin{figure}[hbtp]
  \centering
  \includegraphics[width=0.65\linewidth]{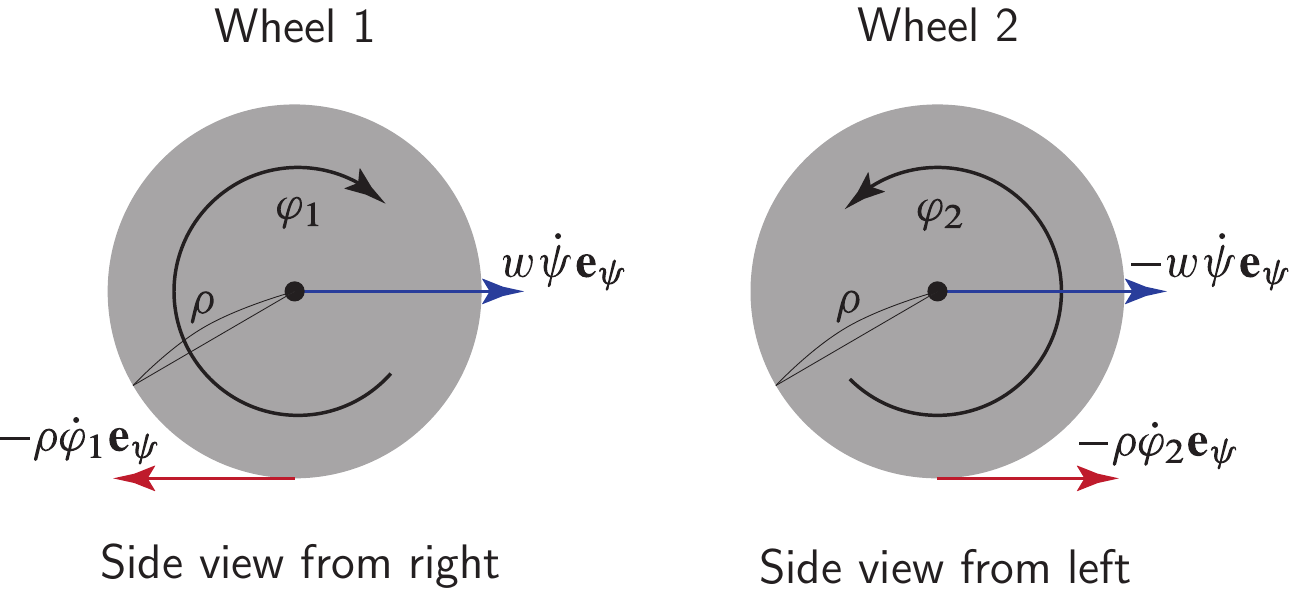}
  \caption{Side views of wheels}
  \label{fig:Wheels}
\end{figure}

On the other hand, the positions of the contact points of the sphere in the body frame is
\begin{equation*}
  \bar{\boldsymbol{q}}_{\rm s}^{(i)}(R,\psi) \defeq R^{T} \mathbf{y}_{\rm w}^{(i)}(\psi)
  \quad
  \text{for}
  \quad
  i = 1, 2,
\end{equation*}
and so the velocities of these contact points in the spatial frame are, using \eqref{eq:dotq},
\begin{equation*}
  \dot{\mathbf{q}}_{\rm s}^{(i)}
  \defeq f\parentheses{ R, \hat{\bomega}, \dot{\mathbf{x}}; \bar{\boldsymbol{q}}_{\rm s}^{(i)}(R,\psi) }
  = \begin{bmatrix}
    \dot{\mathbf{x}} \\
    0
  \end{bmatrix}
  + \hat{\bomega}\mathbf{y}_{\rm w}^{(i)}(\psi).
\end{equation*}
The constraints $\dot{\mathbf{q}}_{\rm s}^{(i)} = \dot{\mathbf{q}}_{\rm w}^{(i)}$ with $i = 1,2$ then yield
\begin{equation*}
  \hat{\bomega}\mathbf{y}_{\rm w}^{(1)}(\psi) = (w\dot{\psi} - \rho\dot{\varphi}_{1})\mathbf{e}_{\psi},
  \qquad
  \hat{\bomega}\mathbf{y}_{\rm w}^{(2)}(\psi) = -(w\dot{\psi} + \rho\dot{\varphi}_{2})\mathbf{e}_{\psi}.
\end{equation*}
However, noting that the hat map~\eqref{eq:hat_map} satisfies, for any $\mathbf{a}, \mathbf{b} \in \R^{3}$,
\begin{equation*}
  \hat{\mathbf{a}} \mathbf{b} = \mathbf{a} \times \mathbf{b} = - \mathbf{b} \times \mathbf{a} = - \hat{\mathbf{b}} \mathbf{a},
\end{equation*}
we have
\begin{equation*}
  -\widehat{\mathbf{y}_{\rm w}^{(1)}(\psi)}\, \bomega = (w\dot{\psi} - \rho\dot{\varphi}_{1})\mathbf{e}_{\psi},
  \qquad
  -\widehat{\mathbf{y}_{\rm w}^{(2)}(\psi)}\, \bomega = -(w\dot{\psi} + \rho\dot{\varphi}_{2})\mathbf{e}_{\psi},
\end{equation*}
or
\begin{equation*}
  Y \bomega = \mathbf{b}
  \quad
  \text{with}
  \quad
  Y \defeq -
  \begin{bmatrix}
    \widehat{\mathbf{y}_{\rm w}^{(1)}(\psi)} \medskip\\
    \widehat{\mathbf{y}_{\rm w}^{(2)}(\psi)}
  \end{bmatrix}
  \in \R^{6 \times 3},
  \quad
  \mathbf{b} \defeq 
  \begin{bmatrix}
    (w\dot{\psi} - \rho\dot{\varphi}_{1})\mathbf{e}_{\psi} \\
    -(w\dot{\psi} + \rho\dot{\varphi}_{2})\mathbf{e}_{\psi}
  \end{bmatrix}.
\end{equation*}
Solving this linear system, we obtain $\bomega = (Y^{T}Y)^{-1}Y^{T} \mathbf{b}$ or
\begin{align}
  \label{eq:constraint-omega12}
  \begin{bmatrix}
    \omega_{1} \\
    \omega_{2}
  \end{bmatrix}
  &= -\frac{\rho}{2h}(\dot{\varphi}_{1} + \dot{\varphi}_{2})
  \begin{bmatrix}
    \cos\psi \\
    \sin\psi
  \end{bmatrix},\\
  \label{eq:constraint-omega3}
  \omega_{3} &= \dot{\psi} - \frac{\rho}{2w}(\dot{\varphi}_{1} - \dot{\varphi}_{2}).
\end{align}

\subsection{Additional Constraint}
\label{ssec:additional_constraint}
We impose one more constraint:
The total angular momentum of the robot about the vertical axis passing though the center of the sphere is conserved.
Recall that the no-slip assumption~\eqref{asmptn:no-slip1} from Section~\ref{ssec:kinematic_model} does not prevent the sphere from rotating about the vertical axis.
Accordingly, we make an (ideal) assumption that there is no friction for such rotations.
Assuming no spinning initially, this amounts to setting the total angular momentum to be zero.
Let $I_{\rm s}$ be the moment of inertia of the sphere about any axis passing through the center (assuming that the mass distribution on the surface of the sphere is homogeneous) and $J$ be that of the electromechanical unit.
Then the constraint is given by
\begin{equation*}
  I_{\rm s} \omega_{3} + J\dot{\psi} = 0.
\end{equation*}
Solving the above constraint equation coupled with \eqref{eq:constraint-omega3} for $\omega_{3}$ and $\dot{\psi}$, we have
\begin{equation}
  \label{eq:omega3-dotvarphi}
  \omega_{3} = -\frac{c J}{I_{\rm s}}(\dot{\varphi}_{1} - \dot{\varphi}_{2}),
\end{equation}
and
\begin{equation}
  \label{eq:dotpsi-dotvarphi}
  \dot{\psi} = c(\dot{\varphi}_{1} - \dot{\varphi}_{2}),
\end{equation}
where we defined
\begin{equation*}
  c \defeq \frac{\rho\,I_{\rm s}}{2w(I_{\rm s} + J)}.
\end{equation*}
Equation~\eqref{eq:dotpsi-dotvarphi} is a holonomic constraint on the variables $(\varphi_{1}, \varphi_{2}, \psi)$ that can be integrated easily:
\begin{equation}
  \label{eq:psi-varphi}
  \psi = c(\varphi_{1} - \varphi_{2}),
\end{equation}
where we set, without loss of generality, $\psi(0) = \varphi_{1}(0) = \varphi_{2}(0) = 0$.
Hence we may eliminate $\psi$ from the formulation by using the holonomic constraint~\eqref{eq:psi-varphi}.
Note that setting $\psi(0) = 0$ means that the $x_{1}$-axis is aligned with the axis of the wheels in the initial configuration; see Fig.~\ref{fig:Configuration}.

\begin{remark}
  The above no-friction assumption for rotations of the sphere about the vertical axis is reasonable if the surface on which the robot is rolling is very smooth.
  One may need to adjust it slightly to take frictions into account depending on how rough the surface is.
  If the surface is very rough, one may assume that the moment of inertia of the sphere $I_{\rm s}$ is much larger than that of the electromechanical unit and so one may take the limit $J/I_{s} \to 0$ to have $c = \rho/(2w)$ instead.
\end{remark}

\section{Geometry of Robot Kinematics}
\label{sec:geometry}
\subsection{Kinematic Control System}
Let us define the configuration space of the robot as
\begin{equation*}
  Q \defeq \mathbb{S}^{1} \times \mathbb{S}^{1} \times \SO(3) \times \R^{2} = \{(\varphi_{1},\varphi_{2},R,\mathbf{x})\}.
\end{equation*}
Then the no-slip constraints \eqref{eq:no-slip} and \eqref{eq:constraint-omega12} along with \eqref{eq:omega3-dotvarphi} define the following nonholonomic constraints on $Q$:
\begin{subequations}
  \label{eq:kinematic_control_system}
  \begin{equation}
    \label{eq:R-omega}
    \dot{R} = \hat{\bomega} R
  \end{equation}
  with
  \begin{equation}
    \label{eq:nonholonomic_constraints}
    \begin{array}{l}
      \DS
      \bomega = 
      \begin{bmatrix}
        \omega_{1} \\
        \omega_{2} \\
        \omega_{3}
      \end{bmatrix}
      =
      \begin{bmatrix}
        -\frac{\rho}{2h}\cos(c(\varphi_{1} - \varphi_{2})) \\
        -\frac{\rho}{2h}\sin(c(\varphi_{1} - \varphi_{2})) \\
        -{c J}/{I_{\rm s}}
      \end{bmatrix}\dot{\varphi}_{1}
      +
      \begin{bmatrix}
        -\frac{\rho}{2h}\cos(c(\varphi_{1} - \varphi_{2})) \\
        -\frac{\rho}{2h}\sin(c(\varphi_{1} - \varphi_{2})) \\
        {c J}/{I_{\rm s}}
      \end{bmatrix}\dot{\varphi}_{2},
      \qquad
      \medskip\\
      \DS
      \dot{\mathbf{x}} = 
      \begin{bmatrix}
        \dot{x}_{1} \\
        \dot{x}_{2}
      \end{bmatrix}
      = \frac{r\rho}{2h}(\dot{\varphi}_{1} + \dot{\varphi}_{2})
      \begin{bmatrix}
        -\sin(c(\varphi_{1} - \varphi_{2})) \\
        \cos(c(\varphi_{1} - \varphi_{2}))
      \end{bmatrix}.
    \end{array}
  \end{equation}
  Assuming that one can control the angular velocity of the wheels, we may define a kinematic control system for the robot by the nonholonomic constraints \eqref{eq:R-omega} and \eqref{eq:nonholonomic_constraints} coupled with
  \begin{equation}
    \label{eq:shape_system}
    \dot{\varphi}_{1} = u_{1},
    \qquad
    \dot{\varphi}_{2} = u_{2}.
  \end{equation}
\end{subequations}
As a result, \eqref{eq:kinematic_control_system} defines a kinematic control system.

\subsection{Geometry of Kinematic Control System}
The above nonholonomic constraints \eqref{eq:R-omega} and \eqref{eq:nonholonomic_constraints} define a distribution $\mathcal{H}$ on $Q$, i.e., at each point $q = (\varphi_{1},\varphi_{2},R,\mathbf{x})$ of $Q$,
\begin{equation*}
  \mathcal{H}_{q} \defeq \setdef{ (\dot{\varphi}_{1},\dot{\varphi}_{2},\dot{R},\dot{\mathbf{x}}) \in T_{q}Q }{ \eqref{eq:R-omega} \text{ and } \eqref{eq:nonholonomic_constraints} }
\end{equation*}
defines a subspace of the tangent space $T_{q}Q$.
Practically speaking, $\mathcal{H}_{q}$ is the space of admissible velocities of the robot at the configuration $q \in Q$.

Now, let $\mathsf{G} \defeq \SO(3) \times \R^{2} = \{ (R,\mathbf{x}) \}$ and $\Phi\colon \mathsf{G} \times Q \to Q$ be the natural (right) action of $\mathsf{G} \defeq \SO(3) \times \R^{2}$ on the $\mathsf{G}$-component of $Q$, i.e.,
\begin{equation}
  \label{eq:Phi}
  \Phi_{(R_{0},\mathbf{x}_{0})}(\varphi_{1}, \varphi_{2}, R, \mathbf{x}) \defeq (\varphi_{1}, \varphi_{2}, R R_{0}, \mathbf{x} + \mathbf{x}_{0}).
\end{equation}
This gives rise to the principal bundle
\begin{equation*}
  \pi\colon Q \to Q/\mathsf{G};
  \quad
  (\varphi_{1}, \varphi_{2}, R, \mathbf{x}) \mapsto (\varphi_{1}, \varphi_{2}),
\end{equation*}
where the base space
\begin{equation*}
  S \defeq Q/\mathsf{G} = \mathbb{S}^{1} \times \mathbb{S}^{1} = \{ (\varphi_{1}, \varphi_{2}) \}  
\end{equation*}
is the so-called \textit{shape space}, i.e., the space of all possible angles of rotation of the two wheels.
Note that $Q$ is a trivial bundle, i.e., $Q = S \times \mathsf{G}$.
In what follows, we will write
\begin{equation*}
  \varphi = (\varphi_{1}, \varphi_{2}) \in S,
  \qquad
  g = (R,\mathbf{x}) \in \mathsf{G},
  \qquad
  q = (\varphi,g) = (\varphi_{1}, \varphi_{2}, R, \mathbf{x}) \in Q
\end{equation*}
for short.
Then we may write the above group action as $\Phi_{g_{0}}(\varphi, g) = (\varphi, g g_{0})$ for any $g_{0} \in \mathsf{G}$.

One can easily show that the distribution $\mathcal{H}$ is invariant under the tangent lift of $\Phi$ in the sense that $T_{q}\Phi_{g}(\mathcal{H}_{q}) = \mathcal{H}_{\Phi_{g}(q)}$ for any $q \in Q$ and any $g \in \mathsf{G}$; in fact, $\hat{\bomega} \defeq \dot{R}R^{-1}$ is clearly invariant under the right action of $\SO(3)$, and the translational symmetry in $\R^{2}$ is trivial.
Let $\mathcal{V}_{q}$ be the tangent space at $q$ to the orbit $\mathcal{O}(q) \defeq \setdef{\Phi_{g}(q)}{g \in \mathsf{G}}$ of the action $\Phi$, i.e., $\mathcal{V}_{q} \defeq T_{q}\mathcal{O}(q)$.
Then it is easy to see that it is a complementary subspace of $\mathcal{H}_{q}$, i.e., $T_{q}Q = \mathcal{H}_{q} \oplus \mathcal{V}_{q}$.
As a result, $\mathcal{H}$ defines a principal connection on $\pi\colon Q \to Q/\mathsf{G}$; see, e.g., \citet{Mo1993a}.

The control system is then defined by the fully actuated subsystem~\eqref{eq:shape_system} in the shape space $S$ coupled with the rest of the system~\eqref{eq:nonholonomic_constraints}---defined by the nonholonomic constraints---in the direction of the fiber $\mathsf{G} = \SO(3) \times \R^{2}$.
A more geometric way of looking at it is the following:
For any given $q = (\varphi,g) \in Q$, we define the horizontal lift $\hl_{q}\colon T_{\varphi}S \to \mathcal{H}_{q}$ as $\hl_{q} \defeq (T_{q}\pi|_{\mathcal{H}_{q}})^{-1}$ or more concretely,
\begin{equation*}
  \hl_{q}(\dot{\varphi}_{1}, \dot{\varphi}_{2}) = (\dot{\varphi}_{1},\dot{\varphi}_{2},\dot{R},\dot{\mathbf{x}})
  \quad\text{with}\quad
  \eqref{eq:R-omega} \text{ and } \eqref{eq:nonholonomic_constraints}.
\end{equation*}
Then the kinematic control system is defined by the horizontal lift of the controlled subsystem~\eqref{eq:shape_system}:
\begin{equation*}
  \dot{q} = \hl_{q}(u_{1}, u_{2}).
\end{equation*}
This is an example of the nonholonomic (kinematic) control system considered by, e.g., \citet{Mo1993a} and \citet{KeMu1995}.

\subsection{Principal Connection Form for Kinematic Control System}
Another way of looking at the above principal connection that is more convenient for our purpose is the following:
We may define a principal connection form $\mathcal{A}\colon TQ \to \mathfrak{g}$ ($\mathfrak{g}$-valued one-form on $Q$), where $\mathfrak{g} = \so(3) \times \R^{2}$ is the Lie algebra of $\mathsf{G} = \SO(3) \times \R^{2}$, as 
\begin{equation*}
  \mathcal{A}_{q} = \mathcal{A}^{\so(3)}_{q} \oplus \mathcal{A}^{\R^{2}}_{q}
\end{equation*}
so that (i)~the distribution $\mathcal{H} \subset TQ$ can be written as $\mathcal{H}_{q} = \ker\mathcal{A}_{q}$; (ii)~it is $\mathsf{G}$-equivariant, i.e., for any $g \in \mathsf{G}$ and $v_{q} \in T_{q}Q$, we have $\mathcal{A}_{q}( T_{q}\Phi_{g}(v_{q}) ) = \Ad_{g^{-1}} \mathcal{A}_{q}(v_{q})$; (iii)~$\mathcal{A}_{q}(\xi_{Q}(q)) = \xi$ for any $\xi \in \mathfrak{g}$, where $\xi_{Q}$ is the infinitesimal generator defined by
\begin{equation}
  \label{eq:inf_gen}
  \xi_{Q}(q) \defeq \left.\od{}{\eps} \Phi_{\exp(\eps\xi)}(q) \right|_{\eps=0}.
\end{equation}
In coordinates, one may write such a connection one-form as (see, e.g., \citet[Proposition~2.9.12 on p.~120]{Bl2015})
\begin{equation}
  \label{eq:A}
  \mathcal{A}_{(\varphi,g)} = \Ad_{g^{-1}}\parentheses{ \d{g}\cdot g^{-1} + A_{i}(\varphi)\,\d\varphi_{i} }.
\end{equation}
More concretely, we may define $\mathcal{A}^{\so(3)}\colon TQ \to \so(3)$ and $\mathcal{A}^{\R^{2}}\colon TQ \to \R^{2}$ as follows:
\begin{align*}
  \mathcal{A}^{\so(3)}_{q}
  \defeq
    \Ad_{R^{-1}}\parentheses{
    \d{R} \cdot R^{-1} + A^{\so(3)}_{i}(\varphi)\,\d\varphi_{i}
  },
  \qquad
  \mathcal{A}^{\R^{2}}_{q}
  \defeq
  \begin{bmatrix}
    \d{x}_{1} \\
    \d{x}_{2}
  \end{bmatrix}
  + \mathbf{A}^{\R^{2}}_{i}(\varphi)\,\d{\varphi}_{i},
\end{align*}
where we used the Einstein summation convention; $\d{R} \cdot R^{-1}$ is seen as an $\so(3)$-valued one-form, i.e., $\d{R} \cdot R^{-1}(\dot{R}) = \dot{R}R^{-1} \in \so(3)$, and
\begin{gather}
  \label{eq:A-so3}
  A^{\so(3)}_{1}(\varphi)
  \defeq 
  \widehat{
    \begin{bmatrix}
      \frac{\rho}{2h}\cos(c(\varphi_{1} - \varphi_{2})) \\
      \frac{\rho}{2h}\sin(c(\varphi_{1} - \varphi_{2})) \\
      {c J}/{I_{\rm s}}
    \end{bmatrix}
  },
  \qquad
  A^{\so(3)}_{2}(\varphi)
  \defeq 
  \widehat{
    \begin{bmatrix}
      \frac{\rho}{2h}\cos(c(\varphi_{1} - \varphi_{2})) \\
      \frac{\rho}{2h}\sin(c(\varphi_{1} - \varphi_{2})) \\
      -{c J}/{I_{\rm s}}
    \end{bmatrix}
  },
  \\
  \label{eq:A-R2}
  \mathbf{A}^{\R^{2}}_{i}(\varphi)
  \defeq \frac{r\rho}{2h}
  \begin{bmatrix}
    \sin(c(\varphi_{1} - \varphi_{2})) \\
    -\cos(c(\varphi_{1} - \varphi_{2}))
  \end{bmatrix}
  \quad
  \text{for}
  \quad
  i = 1, 2, 
\end{gather}
where we used the hat map~\eqref{eq:hat_map}.

As a result, we have $\dot{q} \in \mathcal{H}_{q}$ if and only if $\mathcal{A}_{q}(\dot{q}) = 0$, and the subsystem~\eqref{eq:nonholonomic_constraints} in the direction of the fiber $\SO(3) \times \R^{2}$ can be written as
\begin{equation}
  \label{eq:nonholonomic_constraints-A}
  \hat{\bomega} = \dot{R}R^{-1} = -A^{\so(3)}_{i}(\varphi) \dot{\varphi}_{i},
  \qquad
  \dot{\mathbf{x}} = -\mathbf{A}^{\R^{2}}_{i}(\varphi)\dot{\varphi}_{i}.
\end{equation}
Note that the horizontal lift $\hl_{q}\colon T_{\varphi}S \to \mathcal{H}_{q}$ is then written as
\begin{equation*}
  \hl_{q}(\dot{\varphi}_{1}, \dot{\varphi}_{2})
  = \parentheses{
    \dot{\varphi}_{1}, \dot{\varphi}_{2},
    (-A^{\so(3)}_{i}(\varphi) \dot{\varphi}_{i})R,
    -\mathbf{A}^{\R^{2}}_{i}(\varphi)\dot{\varphi}_{i}
  }.
\end{equation*}

\section{Fiber Controllability}
\label{sec:fiber_controllability}
One of the main questions regarding the kinematic control system~\eqref{eq:kinematic_control_system} is its controllability.
The controllability of the subsystem \eqref{eq:shape_system} in the shape space $S$ is fairly trivial and is of not much practical importance.
What is more important practically is the \textit{fiber controllability}~\cite{KeMu1995}, i.e., the controllability in the direction of the fiber $\mathsf{G} = \SO(3) \times \R^{2}$.
The fiber controllability here addresses the question of whether it is possible to maneuver the robot to an arbitrary (center) position with an arbitrary rotational orientation, regardless of the configurations of the wheels.
As stated in Proposition~4 in \citet{KeMu1995} (see also \citet{Mo1991a,Mo1993a}), the Ambrose--Singer Theorem~\cite{AmSi1953} provides a criterion for fiber controllability in terms of the principal connection $\mathcal{A}$ defined above as well as its curvature.

\subsection{Curvature of Principal Connection}
The nonholonomic/nonintegrable nature of the horizontal distribution $\mathcal{H}$ is essential in the kinematic control system~\eqref{eq:kinematic_control_system}, as the Lie brackets of vector fields in $\mathcal{H}$ then generate directions of motion outside the distribution $\mathcal{H}$.
The lack of integrability is measured by the curvature $\mathcal{B}$ of the principal connection $\mathcal{A}$; it is the $\mathfrak{g}$-valued two-form on $Q$ defined as
\begin{equation*}
  \mathcal{B}(X,Y) \defeq \d\mathcal{A}(\hor X, \hor Y) = -\mathcal{A}([\hor X, \hor Y]),
\end{equation*}
where $X, Y \in T_{q}Q$ and $\hor X, \hor Y \in \mathcal{H}_{q}$ are their horizontal components, i.e.,
\begin{equation*}
  \hor X \defeq X - \parentheses{ \mathcal{A}(X) }_{Q}(q),
\end{equation*}
where $(\,\cdot\,)_{Q}$ stands for the infinitesimal generator defined in \eqref{eq:inf_gen}.
A more convenient formula for $\mathcal{B}$ is given by the Cartan structure equation (see, e.g., \citet[Theorem~2.1.9]{MaMiOrPeRa2007}):
\begin{equation*}
  \mathcal{B}(X,Y) = \d\mathcal{A}(X,Y) + [\mathcal{A}(X), \mathcal{A}(Y)],
\end{equation*}
where we have the plus sign on the right-hand side because $\Phi$, defined in \eqref{eq:Phi}, is a right action.
This formula gives the following coordinate expression for the curvature:
\begin{equation*}
  \mathcal{B}_{q} = \Ad_{g^{-1}}(B(\varphi)\,\d\varphi_{1} \wedge \d\varphi_{2}),
\end{equation*}
where the local expression $B\colon S \to \mathfrak{g}$ of the curvature is written in terms of the local expression $\{A_{i}\colon S \to \mathfrak{g}\}_{i=1,2}$ of the connection one-form \eqref{eq:A} as follows:
Let $\{ e_{a} \}_{a=1}^{\dim\mathsf{G}}$ be a basis for $\mathfrak{g}$, and $A_{i}(\varphi) = A_{i}^{a}(\varphi)\, e_{a}$ for $i = 1, 2$.
Then $B(\varphi) = B^{a}(\varphi)\, e_{a}$ with
\begin{equation*}
  B^{a} = \pd{A_{2}^{a}}{\varphi_{1}} - \pd{A_{1}^{a}}{\varphi_{2}} + C^{a}_{bc} A_{1}^{b} A_{2}^{c},
\end{equation*}
and $C^{a}_{bc}$ is the structure constant of $\mathfrak{g}$ defined as
\begin{equation*}
  [e_{b}, e_{c}] = C^{a}_{bc}e_{a}.
\end{equation*}
More explicitly, we have
\begin{equation*}
  \mathcal{B}_{q} = \mathcal{B}^{\so(3)}_{q} \oplus \mathcal{B}^{\R^{2}}_{q},
\end{equation*}
where
\begin{equation*}
  \mathcal{B}^{\so(3)}_{q} = \Ad_{R^{-1}}(B^{\so(3)}\,\d\varphi_{1} \wedge \d\varphi_{2}),
  \qquad
  \mathcal{B}^{\R^{2}}_{q} = \mathbf{B}^{\R^{2}}\,\d\varphi_{1} \wedge \d\varphi_{2}
\end{equation*}
are the curvatures of the principal connections $\mathcal{A}^{\so(3)}$ and $\mathcal{A}^{\R^{2}}$.
The expressions are then given by
\begin{align}
  B^{\so(3)}
  &= \parentheses{ \pd{(A_{2}^{\so(3)})^{a}}{\varphi_{1}} - \pd{(A_{1}^{\so(3)})^{a}}{\varphi_{2}} + C^{a}_{bc} (A_{1}^{\so(3)})^{b} (A_{2}^{\so(3)})^{c} } \hat{\mathbf{e}}_{a} \nonumber\\
  &= \parentheses{ \pd{\mathbf{A}_{2}^{\so(3)}}{\varphi_{1}} - \pd{\mathbf{A}_{1}^{\so(3)}}{\varphi_{2}} + \mathbf{A}_{1}^{\so(3)} \times \mathbf{A}_{2}^{\so(3)} }^{\widehat{}} \nonumber\\
  &= \frac{\rho^{2}}{2 h w}
  \widehat{
    \begin{bmatrix}
      -\sin(c(\varphi_{1} - \varphi_{2})) \\
      \cos(c(\varphi_{1} - \varphi_{2})) \\
      0
    \end{bmatrix}
  },
  \label{eq:B-so3}
\end{align}
where $\{ \mathbf{e}_{a} \}_{a=1}^{3}$ is the standard basis for $\R^{3}$ and $\mathbf{A}^{\so(3)}_{i}\colon S \to \R^{3}$ is defined so that $\widehat{\mathbf{A}^{\so(3)}_{i}} = A^{\so(3)}_{i}$ for $i = 1, 2$ under the hat map~\eqref{eq:hat_map}, whereas
\begin{equation}
  \label{eq:B-R2}
  \mathbf{B}^{\R^{2}} =
  \pd{\mathbf{A}_{2}^{\R^{2}}}{\varphi_{1}} - \pd{\mathbf{A}_{1}^{\R^{2}}}{\varphi_{2}}
  = \frac{c\,r\rho}{h}
  \begin{bmatrix}
    \cos(c(\varphi_{1} - \varphi_{2})) \\
    \sin(c(\varphi_{1} - \varphi_{2}))
  \end{bmatrix}
\end{equation}
because $\R^{2}$ is abelian.

\subsection{Fiber Controllability}
We are now ready to prove the fiber controllability of the robot:
\begin{theorem}
  \label{thm:fiber_controllability}
  The kinematic control system~\eqref{eq:kinematic_control_system} of the spherical rolling robot is fiber controllable, i.e., given arbitrary two points $g_{0}, g_{1} \in \mathsf{G} = \SO(3) \times \R^{2}$, there exists a control $u\colon[t_{0},t_{1}] \to \R^{2}$ such that the solution $g(t) = (R(t),\mathbf{x}(t))$ of the system~\eqref{eq:kinematic_control_system} under the initial condition $g(t_{0}) = g_{0}$ satisfies $g(t_{1}) = g_{1}$.
\end{theorem}

\begin{proof}
  Let us define, for each $\varphi \in S$, the subspaces $\{ \mathfrak{h}_{i}(\varphi) \}_{i=1}^{2}$ of $\mathfrak{g}$ as follows:
  \begin{equation*}
    \mathfrak{h}_{1}(\varphi) \defeq \Span\{ A(\varphi) \},
    \qquad
    \mathfrak{h}_{2}(\varphi) \defeq \Span\{ B(\varphi) \},
  \end{equation*}
  where $A = A^{\so(3)} \oplus \mathbf{A}^{\R^{2}}$ and $B = B^{\so(3)} \oplus \mathbf{B}^{\R^{2}}$.
  Since $\mathsf{G} = \SO(3) \times \R^{2}$ is a direct product, $\mathfrak{g} = \so(3) \oplus \R^{2}$ is a direct sum.
  Hence we may treat $\so(3)$ and $\R^{2}$ separately.
  First, it is clear from \eqref{eq:A-so3} and \eqref{eq:B-so3} that
  \begin{equation*}
    \Span\{ A^{\so(3)}(\varphi) \} + \Span\{ B^{\so(3)}(\varphi) \} = \so(3),
  \end{equation*}
  and also from \eqref{eq:A-R2} and \eqref{eq:B-R2} that
  \begin{equation*}
    \Span\{ \mathbf{A}^{\R^{2}}(\varphi) \} + \Span\{ \mathbf{B}^{\R^{2}}(\varphi) \} = \R^{2}
  \end{equation*}
  for any $\varphi \in S$.
  This implies that $\mathfrak{h}_{1}(\varphi) \oplus \mathfrak{h}_{2}(\varphi) = \mathfrak{g}$ for any $\varphi \in S$.
  Hence by Proposition~4 from \citet{KeMu1995}, the kinematic control system~\eqref{eq:kinematic_control_system} is locally fiber controllable near any $(\varphi_{0}, g_{0}) \in Q$, i.e., there exists an open neighborhood of $g_{0} \in \mathsf{G}$ that can be reached in the fiber direction.
  
  However, since $\mathsf{G}$ is connected, this implies that, for any $g \in \mathsf{G}$, $(\varphi_{0}, g)$ can be reached from $(\varphi_{0}, g_{0})$.
  In fact, the connectedness implies that, for any $g_{0},g_{1} \in \mathsf{G}$, one can find a path (not the trajectory of the system in general) $g\colon [0,1] \to \mathsf{G}$ such that $g(0) = g_{0}$ and $g(1) = g_{1}$.
  Now the local fiber controllability implies that, for any $t \in [0,1]$, $g(t) \in \mathsf{G}$ has an open neighborhood $U_{t} \subset \mathsf{G}$ that can be reached from $g(t)$; also $g(t)$ can be reached from any point in $U_{t}$ by reversing the control because there is no drift in the system~\eqref{eq:kinematic_control_system}.
  This defines an open cover $\{ U_{t} \}_{t \in [0,1]}$ of the path $g([0,1])$.
  But then, since the path is compact, there exists a finite subcovering, and hence the path is covered by finite open neighborhoods, each of which can be reached from a certain point on the path; also the point on the path can be reached from any point in the neighborhood.
  This proves the existence of a desired control $u$.
\end{proof}

\section{Geometric Phase: Curvature and Locomotion}
\label{sec:curvature_and_locomotion}
The fiber controllability proved above only concerned with existence of a desired control, and does not provide us with a constructive way of finding a desired control.
In this section, we partially address this problem by finding explicit formulas for the changes in the translational position and rotational orientation of the sphere---called \textit{holonomy} or \textit{geometric phase}---under those control laws that result in certain types of loops in the shape space $S$.
These formulas apply to only some special types of control and can generate only certain types of motions, and does not give the control law for \textit{any} maneuver in the fiber.
Nevertheless, these results illustrate how the geometric ingredients introduced above play a role in motion generation, and have potential applications in motion planning; see, e.g., \citet{KeMu1995} and \citet{HaCh2011}.

\subsection{Translational Holonomy and Area Rule}
It is straightforward to calculate the translational motion of the center of the sphere of the robot by integrating the $\R^{2}$ part of \eqref{eq:nonholonomic_constraints}.
A particularly interesting case is where the control $(u_{1},u_{2})$ is applied so that $\varphi = (\varphi_{1},\varphi_{2})$ makes a loop in the shape space torus $S = \mathbb{S}^{1} \times \mathbb{S}^{1}$ or its covering space $\R \times \R$ (if a wheel makes more than one revolution).
In this case, the displacement of the center of the sphere is determined by the weighted area enclosed by the loop determined in terms of the curvature \eqref{eq:B-R2}---an example of the ``area rule'' (see, e.g., \citet{KeMu1995}).

Let $\Gamma\colon [0,T] \to S$ be a loop in the shape space $S$ that encloses a domain $D \subset S$, i.e., $\partial D = \Gamma([0,T])$.
Using Stokes's Theorem, one can find the displacement of the center of the sphere in terms of the curvature as follows:
\begin{align}
  \mathbf{x}(T) - \mathbf{x}(0)
  &= -\int_{\partial D} \mathbf{A}^{\R^{2}}_{i}(\varphi)\d\varphi_{i} \nonumber\\
  &= -\int_{D} \mathbf{B}^{\R^{2}}(\varphi) \d\varphi_{1} \wedge \d\varphi_{2} \nonumber\\
  &= -\int_{D}
    \begin{bmatrix}
      \cos(c(\varphi_{1} - \varphi_{2})) \\
      \sin(c(\varphi_{1} - \varphi_{2}))
    \end{bmatrix}
  \d\varphi_{1} \wedge \d\varphi_{2}.
  \label{eq:area_rule}
\end{align}

\begin{example}
  As a simple and typical example to see the above area rule, consider a rectangular loop $\Gamma$ in the shape space $S = \mathbb{S}^{1} \times \mathbb{S}^{1}$ or its covering space $\R \times \R$ shown in Fig.~\ref{fig:loop-trans_holonomy}.
  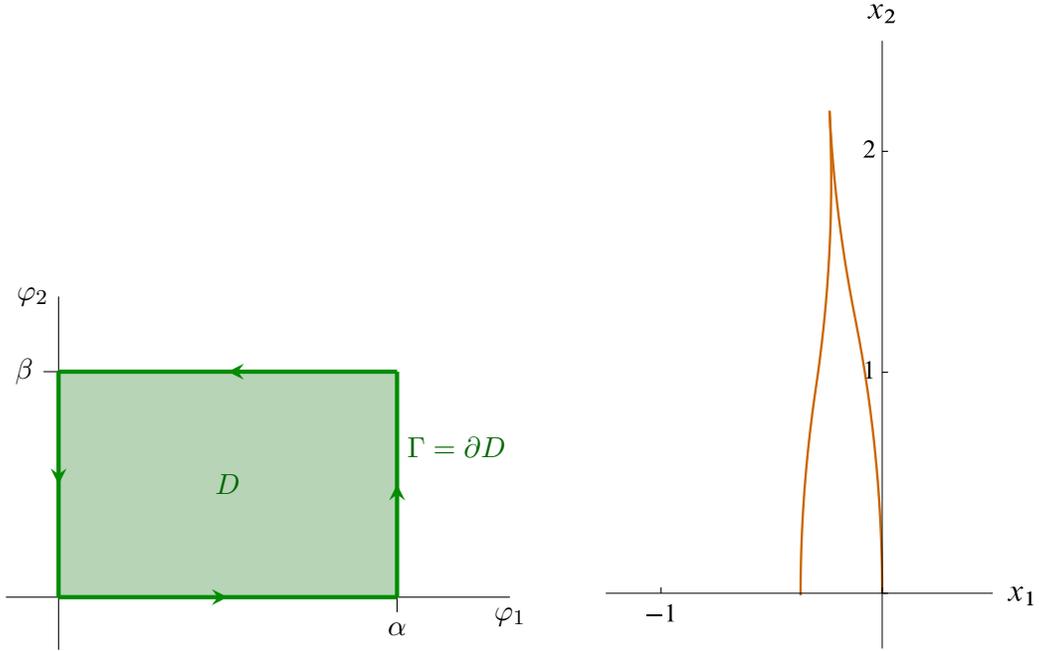
\begin{figure}[htbp]
    \centering
    \subfigure[Loop $\Gamma$ in the shape space $S = \mathbb{S}^{1} \times \mathbb{S}^{1}$ or its covering space $\R \times \R$.]{
      \begin{tikzpicture}[smooth, scale=2, samples=50]
        \clip (-0.35,-0.35) rectangle (3.25,2.25);
        \draw[-] (-0.35,0) -- (3,0) node[below] {$\varphi_{1}$};
        \draw[-] (0,-0.35) -- (0,2) node[left] {$\varphi_{2}$};
        \draw (2.25,0) -- (2.25,-0.1) node[below] {$\alpha$};
        \draw (0,1.5) -- (-0.1,1.5) node[left] {$\beta$};
        \fill[green!40!black!40,opacity=0.7] (0,0) -- (2.25,0) -- (2.25,1.5) -- (0,1.5) -- cycle;
        \draw[-stealth, ultra thick, green!55!black] (0,0) -- (1.125,0);
        \draw[ultra thick, green!55!black] (1.1,0)-- (2.25,0);
        \draw[-stealth, ultra thick, green!55!black] (2.25,0) -- (2.25,0.75);
        \draw[ultra thick, green!55!black] (2.25,0.7) -- (2.25,1.5);
        \draw[-stealth, ultra thick, green!55!black] (2.25,1.5) -- (1.125,1.5);
        \draw[ultra thick, green!55!black] (1.15,1.5) -- (0,1.5);
        \draw[-stealth, ultra thick, green!55!black] (0,1.5) -- (0,0.75);
        \draw[ultra thick, green!55!black] (0,0.8) -- (0,0);
        \node[right, green!40!black] at (2.25,1) {$\Gamma = \partial D$};
        \node[green!40!black] at (1.125,0.75) {$D$};
      \end{tikzpicture}
      \label{fig:loop-trans_holonomy}
    }
    \quad
    \subfigure[Trajectory of the center of the sphere.]{
      \includegraphics[width=.35\linewidth]{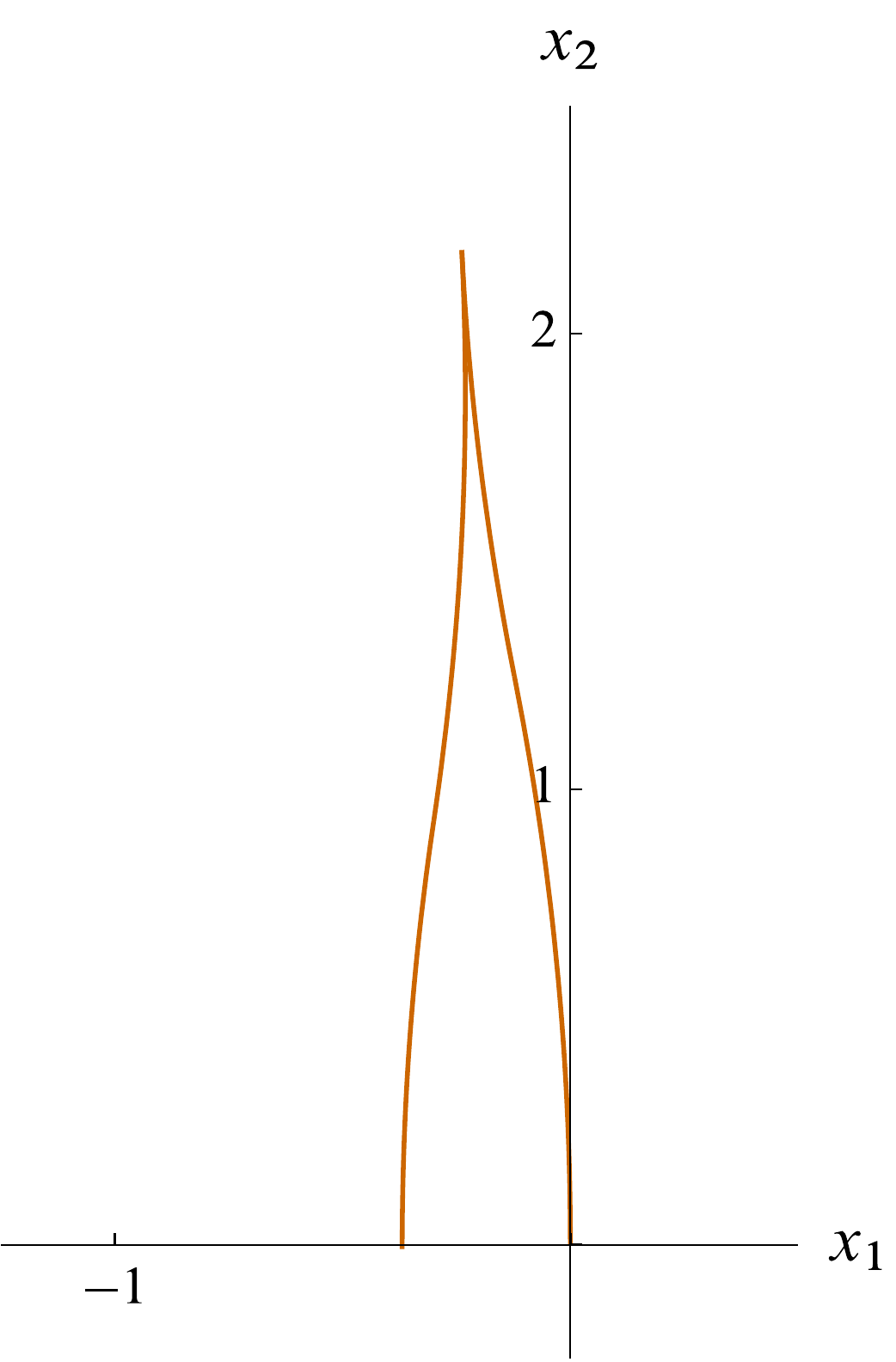}
      \label{fig:x-trans_holonomy}
    }
    \captionsetup{width=0.9\textwidth}
    \caption{Example of translational holonomy by the area rule.
      In (b), the parameters are: $r = 1$, $\rho = 0.3$, $h = 0.75$, $w = 0.8$, and $J/I_{\rm s} = 5$; the initial position is $\mathbf{x}(0) = \mathbf{0}$; the angles are $\alpha = 7\pi$ and $\beta = 6\pi$; the terminal time is $T = 2(\alpha + \beta)$.
    }
    \label{fig:trans_holonomy}
  \end{figure}
  It is straightforward to evaluate the integral over the domain $D$:
  \begin{align*}
    \mathbf{x}(T) - \mathbf{x}(0)
    &= -\int_{D} \mathbf{B}^{\R^{2}}(\varphi) \d\varphi_{1} \wedge \d\varphi_{2} \\
    &= \frac{r \rho}{c\,h}
      \begin{bmatrix}
        \cos(c\alpha) + \cos(c\beta) - \cos\parentheses{ c(\alpha - \beta) } - 1 \\
        \sin(c\alpha) - \sin(c\beta) - \sin\parentheses{ c(\alpha - \beta) } - 1
      \end{bmatrix}.
  \end{align*}
  With the parameters as specified in the caption of Fig.~\ref{fig:trans_holonomy}, the above area rule gives $\mathbf{x}(T) - \mathbf{x}(0) \simeq (-0.37, -0.01)$; this is the actual displacement of the center $\mathbf{x}$ in time $T$ as shown in Fig.~\ref{fig:x-trans_holonomy}.
\end{example}

\subsection{Rotational Holonomy}
How much does the sphere rotate as $\varphi$ makes a loop in the shape space?
Unfortunately, calculation of geometric phases in rotations is not as simple and clear cut as the translational case because of the non-abelian nature of $\SO(3)$.
As we have seen in \eqref{eq:nonholonomic_constraints-A}, the time evolution of the rotational configuration $R \in \SO(3)$ is related to the evolution of the angles $\varphi$ of the wheels as follows:
\begin{equation}
  \label{eq:reconstruction_equation-rot}
 \dot{R} = \parentheses{ -A^{\so(3)}_{i}(\varphi) \dot{\varphi}_{i} } R.
\end{equation}
Suppose that a curve $\varphi\colon [0,T] \to S$ is given.
Then, as is well known in basic theory of linear differential equations, one may formally write down the solution of the above system as an infinite series of integrals.
However, since $\SO(3)$ is non-abelian, this series does not simplify to a matrix exponential in general.
Therefore there is no simple area rule like the (abelian) translational case.

Here we restrict our attention to a particular type of control for which \eqref{eq:reconstruction_equation-rot} is explicitly solvable.
Upon the change of variables to the new coordinates $(\phi_{1}, \phi_{2})$ defined by
\begin{equation}
  \label{eq:phi}
  \phi_{1} \defeq \varphi_{1} + \varphi_{2},
  \qquad
  \phi_{2} \defeq \varphi_{1} - \varphi_{2},
\end{equation}
the connection form $\mathcal{A}^{\so(3)}$ becomes
\begin{align*}
  \mathcal{A}^{\so(3)}_{q}
  \defeq
    \Ad_{R^{-1}}\parentheses{
    \d{R} \cdot R^{-1} + \tilde{A}^{\so(3)}_{i}(\phi)\,\d\phi_{i}
  },
\end{align*}
where
\begin{gather*}
  \tilde{A}^{\so(3)}_{1}(\phi)
  \defeq
  \frac{\rho}{2h}
  \widehat{
    \begin{bmatrix}
      \cos(c\,\phi_{2}) \\
      \sin(c\,\phi_{2}) \\
      0
    \end{bmatrix}
  },
  \qquad
  \tilde{A}^{\so(3)}_{2}(\phi)
  \defeq
  \frac{c J}{I_{\rm s}}
  \widehat{
    \begin{bmatrix}
      0 \\
      0 \\
      1
    \end{bmatrix}
  },
\end{gather*}
Note that the first one is constant if $\phi_{2}$ is constant, while the second one is always constant.
This suggests us to make a loop in the shape space so that each edge is parallel to either the $\phi_{1}$- or $\phi_{2}$-axis; a typical loop of this type is shown in Fig.~\ref{fig:loop-rot_holonomy} in the $\varphi_{1}$-$\varphi_{2}$ plane.
Particularly, if we require that the angular velocities are piecewise constant, the loop is given by
\begin{equation*}
  \varphi_{1}(t) =
  \begin{cases}
    t/2 & 0 \le t < \alpha + \beta, \\
    \alpha + \beta - t/2 & \alpha + \beta \le t \le 2(\alpha + \beta),
  \end{cases}
  \quad
  \varphi_{2}(t) =
  \begin{cases}
    t/2 & 0 \le t < \alpha, \\
    \alpha - t/2 & \alpha \le t < 2\alpha + \beta, \\
    t/2 - (\alpha + \beta) & 2\alpha + \beta \le t \le 2(\alpha + \beta),
  \end{cases}
\end{equation*}
under the piecewise constant control
\begin{equation}
  \label{eq:u-rot_holonomy}
  u(t) = (u_{1}(t),u_{2}(t)) =
  \begin{cases}
    (1/2, 1/2) & 0 \le t < \alpha, \\
    (1/2,-1/2) & \alpha \le t < \alpha + \beta, \\
    (-1/2,-1/2) & \alpha + \beta \le t < 2\alpha + \beta, \\
    (-1/2,1/2) & 2\alpha + \beta \le t \le 2(\alpha + \beta).
  \end{cases}
\end{equation}

\begin{figure}
  \centering
  \subfigure[Loop in the shape space for rotational holonomy.]{
    \begin{tikzpicture}[smooth, scale=2, samples=50]
      \clip (-0.35,-1.5) rectangle (3.5,1.75);
      \draw[-] (-0.35,0) -- (3,0) node[below] {$\varphi_{1}$};
      \draw[-] (0,-1.25) -- (0,1.5) node[left] {$\varphi_{2}$};
      \fill[green!40!black!40,opacity=0.7] (0,0) -- (1.25,1.25) -- (2,0.5) -- (0.75,-0.75) -- cycle;
      \draw[-stealth, ultra thick, green!55!black] (0,0) -- (0.625,0.625);
      \draw[ultra thick, green!55!black] (0.6,0.6) -- (1.25,1.25) node[above, green!40!black] {$\left(\frac{\alpha}{2},\frac{\alpha}{2}\right)$};
      \draw[-stealth, ultra thick, green!55!black] (1.25,1.25) -- (1.625,0.875);
      \draw[ultra thick, green!55!black] (1.6,0.9) -- (2,0.5) node[right, green!40!black] {$\left(\frac{\alpha+\beta}{2},\frac{\alpha-\beta}{2}\right)$};
      \draw[-stealth, ultra thick, green!55!black] (2,0.5) -- (1.375,-0.125);
      \draw[ultra thick, green!55!black] (1.4,-0.1) -- (0.75,-0.75) node[below, green!40!black] {$\left(\frac{\beta}{2},\frac{\beta}{2}\right)$};
      \draw[-stealth, ultra thick, green!55!black] (0.75,-0.75) -- (0.375,-0.375);
      \draw[ultra thick, green!55!black] (0.4,-0.4) -- (0,0);
      \node[right, green!40!black] at (1.625,1) {$\Gamma$};
      \node[green!40!black] at (1,0.25) {$D$};
    \end{tikzpicture}
    \label{fig:loop-rot_holonomy}
  }
  \quad
  \subfigure[Trajectory of a point on the sphere.]{
    \includegraphics[width=.425\linewidth]{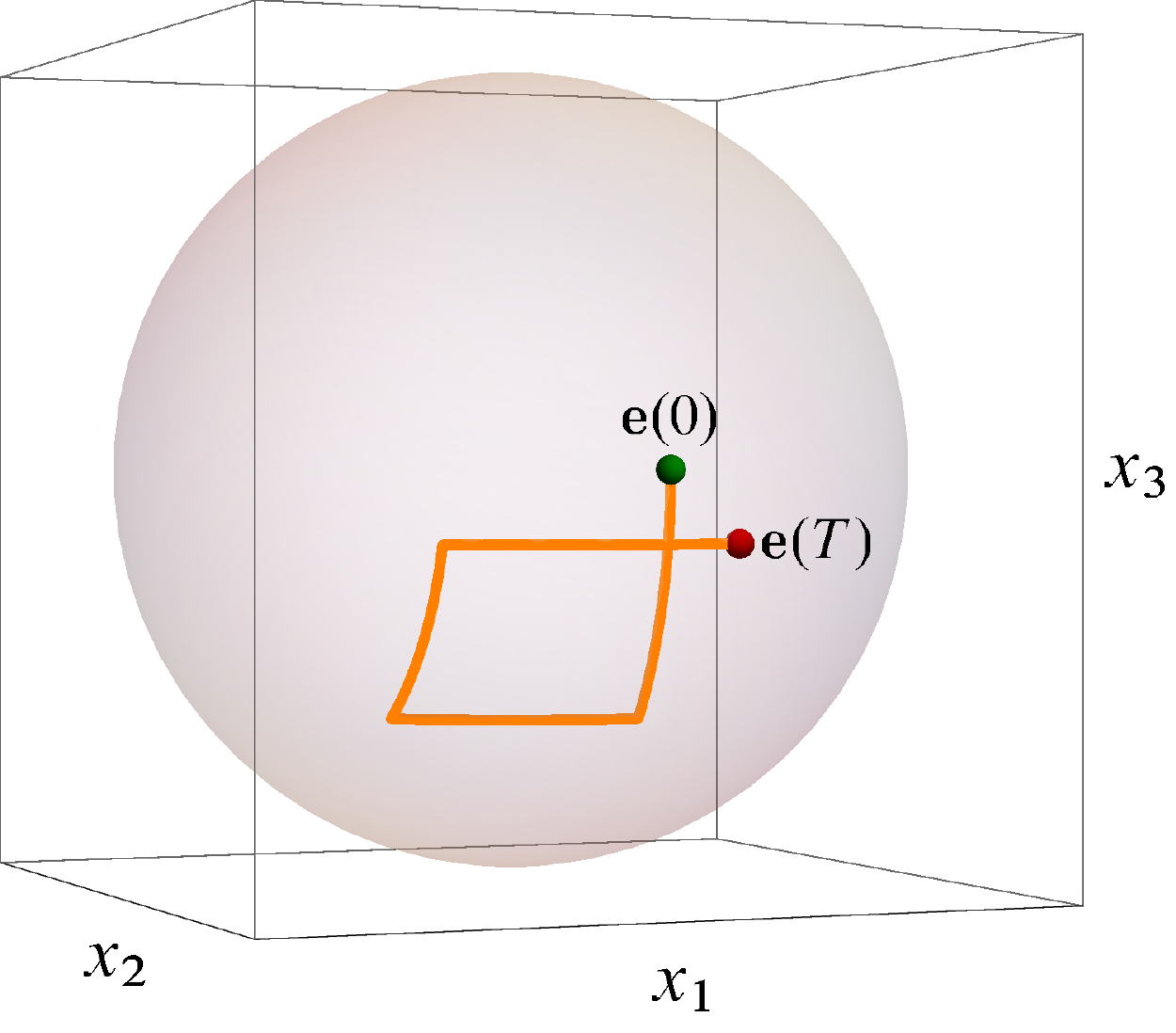}
    \label{fig:rot_holonomy}
  }
  \captionsetup{width=0.9\textwidth}
  \caption{(a)~A loop along which one can find an exact expression for the rotational holonomy.
    (b)~The initial position of the point is $\mathbf{e}(0) = (0,1,0)$.
    The parameters are the same as those from the caption of Fig.~\ref{fig:trans_holonomy} except that $\alpha = \pi$ and $\beta = 3\pi/2$ here.
  }
\end{figure}
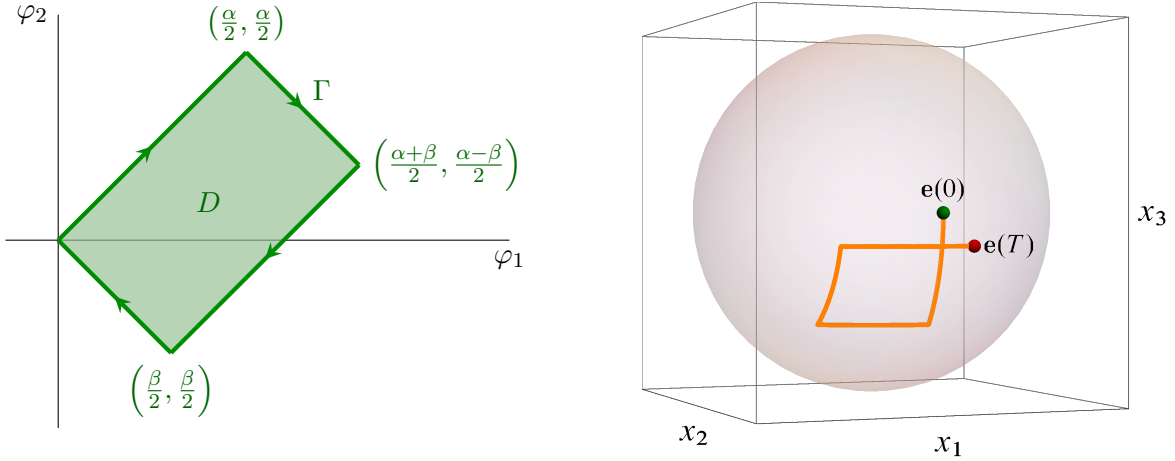

Then we can compute the rotational holonomy explicitly because \eqref{eq:reconstruction_equation-rot} is exactly solvable along each edge of the rectangular loop:
\begin{align*}
  R(\alpha) &= \exp\parentheses{
              -\alpha\,\frac{\rho}{2h}
              \widehat{
              \begin{bmatrix}
                1 \\
                0 \\
                0
              \end{bmatrix}}
  } R(0), &
            R(\alpha+\beta) &= \exp\parentheses{
                              -\beta\,\frac{c J}{I_{\rm s}}
                              \widehat{
                              \begin{bmatrix}
                                0 \\
                                0 \\
                                1
                              \end{bmatrix}}
  } R(\alpha),
  \\
  R(2\alpha + \beta) &= \exp\parentheses{
                       \alpha\,\frac{\rho}{2h}
                       \widehat{
                       \begin{bmatrix}
                         \cos(c\beta) \\
                         \sin(c\beta) \\
                         0
                       \end{bmatrix}}
  } R(\alpha + \beta), &
                         R(2(\alpha+\beta)) &= \exp\parentheses{
                                              \beta\,\frac{c J}{I_{\rm s}}
                                              \widehat{
                                              \begin{bmatrix}
                                                0 \\
                                                0 \\
                                                1
                                              \end{bmatrix}}
  } R(2\alpha + \beta).
\end{align*}
Combining the above results, we obtain an explicit expression for the rotational holonomy $R(2(\alpha+\beta))$ (assuming $R(0) = I$ without of loss of generality) picked up after the loop in Fig.~\ref{fig:loop-rot_holonomy} is traversed.

Note that the resulting rotational holonomy is independent of a particular choice of control as long as the curve traverses the same loop.
In other words, one obtains the same holonomy $R(2(\alpha+\beta))$ along any curve $\varphi\colon [0,T] \to S$ that traverses the loop in the shape space as shown in Fig.~\ref{fig:loop-rot_holonomy}.

Fig.~\ref{fig:rot_holonomy} shows an example of the trajectory $\mathbf{e}(t)$ in space (modulo translations of the center of the sphere) for $0 \le t \le T = 2(\alpha + \beta)$ of a point fixed on the sphere under the control~\eqref{eq:u-rot_holonomy}, that is, $\mathbf{e}(t) = R(t)\,\mathbf{e}(0)$.

Particularly, if we pick $\beta = 2\pi/c$, then the translational holonomy vanishes because the curvature of the translational part of the connection in terms of $(\phi_{1},\phi_{2})$ is
\begin{equation*}
  \mathbf{B}^{\R^{2}}
  = -\frac{c\,r \rho}{2h}
  \begin{bmatrix}
    \cos(c\,\phi_{2}) \\
    \sin(c\,\phi_{2})
  \end{bmatrix}\,\d\phi_{1} \wedge \d\phi_{2},
\end{equation*}
and vanishes if integrated over $0\le \phi_{2} \le \beta = 2\pi/c$.
Hence $\mathbf{x}(T) = \mathbf{x}(0)$ by the area rule~\eqref{eq:area_rule}.
Therefore, with this particular choice of $\beta$, the center of the robot comes back to the original position but picks up the rotational phase calculated above, i.e., the total geometric phase is only rotational.

\section{Optimal Control between Two Center Positions}
\label{sec:optimal_control}
Let us now consider an optimal control problem of the robot.
We will restrict ourselves to a simple special case: The terminal time is fixed, and only the wheel angles and the translational configurations (the position of the center of the sphere) are specified at the end (initial and terminal) times, i.e., the rotational configurations at the end times are immaterial.
We show that the resulting optimal control system is completely integrable, and obtain an explicit solution for it using Jacobi's elliptic function.

\subsection{Sub-Riemannian Geodesic and Optimal Control Problem}
Consider the problem of maneuvering the robot from a given position of the center of the sphere to another (regardless of the rotational orientations) in the most ``efficient'' way.
Here we measure the efficiency in terms of the ``energy'' of a curve $q\colon [0,T] \to S \times \R^{2}$ defined by
\begin{equation*}
  E(q) \defeq \int_{0}^{T} \frac{1}{2} \norm{ \dot{\varphi} }^{2}\,dt
  = \int_{0}^{T} \frac{1}{2}\parentheses{ u_{1}(t)^{2} + u_{2}(t)^{2} }\,dt,
\end{equation*}
where $\norm{\,\cdot\,}$ is the standard Euclidean norm, i.e., $\norm{\dot{\varphi}(t)} \defeq \sqrt{ \dot{\varphi}_{1}(t)^{2} + \dot{\varphi}_{2}(t)^{2} }$.
The curves are subject to the condition that they are horizontal, i.e., satisfy the nonholonomic constraint~\eqref{eq:nonholonomic_constraints-A}, and join two given points $q_{0} \defeq (\varphi(0),\mathbf{x}(0))$ and $q_{T} \defeq (\varphi(T),\mathbf{x}(T))$ in $S \times \R^{2}$.
More specifically, we have the following optimal control problem:
\begin{equation}
  \label{eq:OCP1}
  \DS \min_{u} \int_{0}^{T} \frac{1}{2}(u_{1}(t)^{2} + u_{2}(t)^{2})\,dt \medskip\\
  \quad\text{subject to}\quad
  \left\{
    \begin{array}{l}
      \dot{\mathbf{x}} = -\mathbf{A}^{\R^{2}}_{i}(\varphi)\,u_{i}, \smallskip\\
      \dot{\varphi}_{1} = u_{1}, \quad \dot{\varphi}_{2} = u_{2}, \smallskip\\
      \text{$\mathbf{x}(0)$ and $\mathbf{x}(T)$ fixed}, \smallskip\\
      \text{$\varphi(0) = \varphi(T)$ fixed}.
    \end{array}
  \right.
\end{equation}

Those curves that minimize this particular form of energy are intimately related to the so-called sub-Riemannian geodesics:
Let us define the length of a curve $q\colon [0,T] \to S \times \R^{2}$ connecting $q_{0}$ and $q_{T}$ by
\begin{equation*}
  \ell(q) \defeq \int_{0}^{T} \norm{ \dot{\varphi} }\,dt
  = \int_{0}^{T} \sqrt{ u_{1}(t)^{2} + u_{2}(t)^{2} } \,dt.
\end{equation*}
Note that the metric used here is degenerate because the length is measured in terms of only $\dot{\varphi}$ in the derivative $\dot{q} = (\dot{\varphi},\dot{\mathbf{x}})$.
A curve that minimizes such a length is called a sub-Riemannian geodesic; see, e.g., \citet{Mo2002}.
In this particular setting, it is the shortest path in the shape space $S$ whose horizontal lift to $S \times \R^{2}$ joins $q_{0}$ and $q_{T}$.
What one can show (see, e.g., \cite[Proposition~1.6]{Mo2002}) is that $\varphi$ is a minimizer of the energy $E$ if and only if $\varphi$ is a sub-Riemannian geodesic with constant speed, i.e., $\norm{ \dot{\varphi}(t) } = \text{const}.$\footnote{Note that one can reparametrize a curve $\varphi\colon [0,\tilde{T}] \to S$ with $\norm{\dot{\varphi}(\tilde{t})} \neq \text{const}.$ by its arc length $t$ (or constant multiple of it) so that $\norm{\dot{\varphi}(t)} = \text{const}$.}.

Moreover, as mentioned in \citet{SaMo1992}, this type of optimal control problem is also related to the time-optimal control of the same system:
The normalized control $u(t)/\norm{u(t)}$ solves the time-optimal control problem to minimize the time $T$ subject to the same system as above as well as the constraint $\norm{u(t)} \le 1$ on the control inputs.

\subsection{Pontryagin Maximum Principle}
Let us write
\begin{equation*}
  q = (\varphi, \mathbf{x}) \in S \times \R^{2},
  \qquad
  p = (\gamma, \mathbf{p}) \in T_{(\varphi,\mathbf{x})}^{*}(S \times \R^{2}).
\end{equation*}
Then the control Hamiltonian is defined by
\begin{equation*}
  H_{\text{c}}(q, p, u)
  \defeq \mathbf{p} \cdot \parentheses{ -\mathbf{A}^{\R^{2}}_{i}(\varphi) u_{i} } + \gamma \cdot u - \frac{1}{2}(u_{1}^{2} + u_{2}^{2}).
\end{equation*}
Since it is quadratic in control $u$, it is easily maximized with respect to $u$ to yield the optimal control $u^{\star}(q,p) = \argmax_{u \in \R^{2}} H_{\text{c}}(q, p, u)$; specifically,
\begin{align*}
  u^{\star}_{i}(\varphi,\gamma)
  &= \gamma_{i} - \mathbf{p} \cdot \mathbf{A}^{\R^{2}}_{i}(\varphi) \\
  &= \gamma_{i} - \frac{r\rho}{2h} \brackets{ p_{1} \sin (c (\varphi_{1}-\varphi_{2})) - p_{2} \cos (c (\varphi_{1}-\varphi_{2})) }
\end{align*}
for $i = 1, 2$.
Hence we have the Hamiltonian
\begin{align}
  H(q,p)
  &\defeq \max_{u \in \R^{2}} H_{\text{c}}(q, p, u)
  \nonumber \\
  &= \frac{1}{2}\parentheses{ u^{\star}_{1}(\varphi,\gamma)^{2} + u^{\star}_{2}(\varphi,\gamma)^{2} }
  = \frac{1}{2}\parentheses{ \gamma_{i} - \mathbf{p} \cdot \mathbf{A}^{\R^{2}}_{i}(\varphi) }^{2}.
  \label{eq:H}
\end{align}
Then the optimal solution necessarily satisfies the Hamiltonian system
\begin{equation*}
  \dot{q} = \pd{H}{p},
  \qquad
  \dot{p} = -\pd{H}{q},
\end{equation*}
or
\begin{equation*}
  \begin{array}{c}
    \DS \dot{\mathbf{x}} = -\mathbf{A}^{\R^{2}}_{i}(\varphi)\,u^{\star}_{i}(\varphi,\gamma),
    \qquad
    \dot{\varphi}_{i} = u^{\star}_{i}(\varphi,\gamma),
    \medskip\\
    \DS \dot{\mathbf{p}} = \mathbf{0},
    \qquad
    \dot{\gamma}_{i} = \mathbf{p} \cdot \pd{\mathbf{A}^{\R^{2}}_{j}}{\varphi_{i}} u^{\star}_{j}(\varphi,\gamma).
  \end{array}
\end{equation*}
More explicitly, we have
\begin{equation}
  \label{eq:OptimalSystem}
  \begin{array}{c}
    \DS \dot{\mathbf{x}} = -\frac{r\rho}{2h} \parentheses{ u^{\star}_{1}(\varphi,\gamma) + u^{\star}_{2}(\varphi,\gamma) }
    \begin{bmatrix}
      -\sin( c(\varphi_{1} - \varphi_{2}) ) \\
      \cos( c(\varphi_{1} - \varphi_{2}) ) 
    \end{bmatrix},
    \qquad
    \dot{\varphi}_{i} = u^{\star}_{i}(\varphi,\gamma),
    \qquad
    \DS \dot{\mathbf{p}} = \mathbf{0},
    \medskip\\
    \DS
    \begin{bmatrix}
      \dot{\gamma}_{1} \\
      \dot{\gamma}_{2}
    \end{bmatrix}
    = c \frac{r\rho}{2h}
    \parentheses{ p_{1} \cos( c(\varphi_{1} - \varphi_{2}) ) + p_{2} \sin( c(\varphi_{1} - \varphi_{2}) ) }
    \parentheses{ u^{\star}_{1}(\varphi,\gamma) + u^{\star}_{2}(\varphi,\gamma) }
    \begin{bmatrix}
      1 \\
      -1
    \end{bmatrix}.
  \end{array}
\end{equation}

\subsection{Symmetry and Integrability of Optimal Solution}
The Hamiltonian~\eqref{eq:H} is clearly independent of $\mathbf{x}$ and hence the corresponding costate $\mathbf{p} = (p_{1}, p_{2})$ is conserved as one sees in \eqref{eq:OptimalSystem}.
The system also has the following $\mathbb{S}^{1}$ symmetry:
Define an $\mathbb{S}^{1}$ action
\begin{equation*}
  \mathbb{S}^{1} \times (S \times \R^{2}) \to S \times \R^{2};
  \qquad
  (\varphi_{0}, (\varphi_{1}, \varphi_{2}, \mathbf{x})) \mapsto (\varphi_{1} + \varphi_{0}, \varphi_{2} + \varphi_{0}, \mathbf{x}).
\end{equation*}
Its cotangent lift is
\begin{equation*}
  \mathbb{S}^{1} \times T^{*}(S \times \R^{2}) \to T^{*}(S \times \R^{2});
  \qquad
  (\varphi_{0}, (\varphi_{1}, \varphi_{2}, \mathbf{x}, \gamma_{1}, \gamma_{2}, \mathbf{p}))
  \mapsto (\varphi_{1} + \varphi_{0}, \varphi_{2} + \varphi_{0}, \mathbf{x}, \gamma_{1}, \gamma_{2}, \mathbf{p}),
\end{equation*}
and the Hamiltonian~\eqref{eq:H} is invariant under this action.
As a result, $\gamma_{1} + \gamma_{2}$ is conserved as well; this is easy to see directly in \eqref{eq:OptimalSystem} as well.
One may also set, as in \eqref{eq:phi},
\begin{equation}
  \label{eq:phi-sigma}
  \phi_{1} \defeq \varphi_{1} + \varphi_{2},
  \qquad
  \phi_{2} \defeq \varphi_{1} - \varphi_{2},
  \qquad
  \sigma_{1} \defeq \frac{\gamma_{1} + \gamma_{2}}{2},
  \qquad
  \sigma_{2} \defeq \frac{\gamma_{1} - \gamma_{2}}{2}
\end{equation}
so that $(\varphi,\gamma) \mapsto (\phi,\sigma)$ is a canonical change of coordinates.
Then one easily sees that the Hamiltonian~\eqref{eq:H} is independent of $\phi_{1}$ and hence $\sigma_{1}$ is conserved.

Now, the Poisson bracket on $T^{*}(S \times \R^{2})$ is defined as follows: For any $F, G \in C^{\infty}(T^{*}(S \times \R^{2}))$,
\begin{equation*}
  \PB{F}{G} \defeq \pd{F}{q} \cdot \pd{G}{p} - \pd{G}{q} \cdot \pd{F}{p}.
\end{equation*}
It is straightforward to see that the four first integrals
\begin{equation*}
  F_{1} \defeq H,
  \quad
  F_{2} \defeq \gamma_{1} + \gamma_{2},
  \quad
  F_{3} \defeq p_{1},
  \quad
  F_{4} \defeq p_{2}
\end{equation*}
are independent, and also are in involution, i.e., $\PB{F_{i}}{F_{j}} = 0$ for $i, j = 1, \dots 4$.
Hence the system is completely integrable.

\subsection{Exact Solution}
In order to obtain an exact solution to the above system, we reduce it to the equation for a nonlinear pendulum.
To that end, we exploit some of the first integrals from above to rewrite the system~\eqref{eq:OptimalSystem}.

Let us first use $(F_{3}, F_{4}) = (p_{1},p_{2}) = \mathbf{p}$.
Since $\mathbf{p}$ is conserved, we may set
\begin{equation*}
  \mathbf{p} = (p_{1},p_{2}) = |\mathbf{p}| (\cos\delta, \sin\delta),
\end{equation*}
where $|\mathbf{p}|$ and $\delta$ are both constant.
Now let us set
\begin{align*}
  \tilde{\gamma}_{i}
  &\defeq u^{\star}_{i}(\varphi,\gamma) \\
  &= \gamma_{i} - \frac{r\rho}{2h} \parentheses{ p_{1} \sin(c\,\phi_{2}) - p_{2} \cos(c\,\phi_{2}) } \\
  &= \gamma_{i} - \frac{r\rho}{2h} |\mathbf{p}| \sin(c\,\phi_{2} - \delta)
\end{align*}
for $i = 1, 2$.
Then we can write the Hamiltonian $H$ in terms of them as
\begin{equation}
  \label{eq:H-tildegamma}
  H(q,p) = \frac{1}{2} \tilde{\gamma}_{i}^{2}
  = \frac{1}{4}\parentheses{ (\tilde{\gamma}_{1} + \tilde{\gamma}_{2})^{2} + (\tilde{\gamma}_{1} - \tilde{\gamma}_{2})^{2} },
\end{equation}
and also the differential equations for the angles $(\phi_{1}, \phi_{2})$ as
\begin{equation*}
  \dot{\phi}_{1} = \tilde{\gamma}_{1} + \tilde{\gamma}_{2},
  \qquad
  \dot{\phi}_{2} = \tilde{\gamma}_{1} - \tilde{\gamma}_{2}.
\end{equation*}
However, the above expression~\eqref{eq:H-tildegamma} of the Hamiltonian $H$ motivates us to set
\begin{subequations}
  \begin{align}
    \tilde{\gamma}_{1} + \tilde{\gamma}_{2} &= \gamma_{1} + \gamma_{2} - \frac{r\rho}{h} |\mathbf{p}| \sin(c\,\phi_{2} - \delta) = 2\sqrt{H} \cos\theta, \label{eq:tildegamma1+2} \\
    \tilde{\gamma}_{1} - \tilde{\gamma}_{2} &= \gamma_{1} - \gamma_{2} = 2\sqrt{H} \sin\theta
  \end{align}
\end{subequations}
using a new variable $\theta$ so that we have
\begin{equation}
  \label{eq:phi-diffeq}
  \dot{\phi}_{1} = 2\sqrt{H} \cos\theta,
  \qquad
  \dot{\phi}_{2} = 2\sqrt{H} \sin\theta.
\end{equation}
Let us obtain a differential equation for $\theta$.
First observe that
\begin{equation*}
  -\dot{\theta} \csc^{2}\theta
  = \od{}{t} \cot\theta
  = \od{}{t} \parentheses{ \frac{\dot{\phi}_{1}}{\dot{\phi}_{2}} }
  = \od{}{t} \parentheses{ \frac{\tilde{\gamma}_{1} + \tilde{\gamma}_{2}}{\tilde{\gamma}_{1} - \tilde{\gamma}_{2}} }.
\end{equation*}
Let us evaluate the right-hand side.
First rewrite \eqref{eq:H-tildegamma} as
\begin{equation*}
  \frac{1}{2}\parentheses{ \frac{\tilde{\gamma}_{1} + \tilde{\gamma}_{2}}{\tilde{\gamma}_{1} - \tilde{\gamma}_{2}} }^{2}
  = \frac{2H}{(\tilde{\gamma}_{1} - \tilde{\gamma}_{2})^{2}} - \frac{1}{2}.
\end{equation*}
Taking the time derivative of both sides, we have
\begin{equation*}
  (\tilde{\gamma}_{1} + \tilde{\gamma}_{2}) \cdot \od{}{t} \parentheses{ \frac{\tilde{\gamma}_{1} + \tilde{\gamma}_{2}}{\tilde{\gamma}_{1} - \tilde{\gamma}_{2}} }
  = -\frac{4H}{(\tilde{\gamma}_{1} - \tilde{\gamma}_{2})^{2}}\parentheses{ \dot{\tilde{\gamma}}_{1} - \dot{\tilde{\gamma}}_{2} }.
\end{equation*}
However, since $\tilde{\gamma}_{1} - \tilde{\gamma}_{2} = \gamma_{1} - \gamma_{2}$, we have, using \eqref{eq:OptimalSystem},
\begin{align*}
  \dot{\tilde{\gamma}}_{1} - \dot{\tilde{\gamma}}_{2}
  &= \dot{\gamma}_{1} - \dot{\gamma}_{2} \\
  &= c \frac{r\rho}{h} \parentheses{ \gamma_{1} + \gamma_{2} - \frac{r\rho}{h} (p_{1}\sin(c\,\phi_{2}) - p_{2}\cos(c\,\phi_{2})) } (p_{1}\cos(c\,\phi_{2}) + p_{2}\sin(c\,\phi_{2}) ) \\
  &= c \frac{r\rho}{h}|\mathbf{p}| \parentheses{ \tilde{\gamma_{1}} + \tilde{\gamma_{2}} } \cos(c\,\phi_{2} - \delta).
\end{align*}
Therefore,
\begin{align*}
  \od{}{t} \parentheses{ \frac{\tilde{\gamma}_{1} + \tilde{\gamma}_{2}}{\tilde{\gamma}_{1} - \tilde{\gamma}_{2}} }
  &= -\frac{4H}{(\tilde{\gamma}_{1} - \tilde{\gamma}_{2})^{2}} \cdot c \frac{r\rho}{h}|\mathbf{p}| \cos(c\,\phi_{2} - \delta) \\
  &= -\csc^{2}\theta \cdot c \frac{r\rho}{h}|\mathbf{p}| \cos(c\,\phi_{2} - \delta).
\end{align*}
As a result, we obtain
\begin{equation*}
  \dot{\theta} = c \frac{r\rho}{h}|\mathbf{p}| \cos(c\,\phi_{2} - \delta).
\end{equation*}
On the other hand \eqref{eq:tildegamma1+2} gives
\begin{equation*}
  \frac{r\rho}{h} |\mathbf{p}| \sin(c\,\phi_{2} - \delta)
  = \gamma_{1} + \gamma_{2} - 2\sqrt{H} \cos\theta
  = 2\parentheses{ \sigma_{1} - \sqrt{H} \cos\theta },
\end{equation*}
where we used the definition of $\sigma_{1}$ from \eqref{eq:phi-sigma} as well.
Setting
\begin{equation*}
  a \defeq \frac{r\rho}{h} |\mathbf{p}|,
\end{equation*}
we have
\begin{subequations}
  \begin{align}
    \dot{\theta} &= c\,a \cos(c\,\phi_{2} - \delta), \\
    2\parentheses{ \sqrt{H} \cos\theta - \sigma_{1}} &= -a \sin(c\,\phi_{2} - \delta), \label{eq:a-H-sigma_2}
  \end{align}
\end{subequations}
and thus we have
\begin{equation*}
  \dot{\theta}^{2} + 4c^{2}\parentheses{ \sqrt{H} \cos\theta - \sigma_{1} }^{2} = c^{2} a^{2}
\end{equation*}
or
\begin{equation*}
  \parentheses{ \od{\theta}{t} }^{2}
  = c^{2}\parentheses{ a^{2} - 4\parentheses{ \sqrt{H} \cos\theta - \sigma_{1} }^{2} }.
\end{equation*}

Now, assuming $\theta \in (-\pi, \pi)$, we introduce a new variable $\vartheta \in (-\pi, \pi)$ defined by
\begin{equation*}
  \tan(\vartheta/2)
  = \sqrt{ \frac{a + 2(\sqrt{H} + \sigma_{1})}{a - 2(\sqrt{H} - \sigma_{1})} } \tan(\theta/2)
\end{equation*}
or
\begin{equation}
  \label{eq:vartheta}
  \vartheta \defeq 2 \arctan \parentheses{ \sqrt{ \frac{a + 2(\sqrt{H} + \sigma_{1})}{a - 2(\sqrt{H} - \sigma_{1})} } \tan(\theta/2) },
\end{equation}
where we also assumed that $a - 2(\sqrt{H} - \sigma_{1}) > 0$; note that $a + 2(\sqrt{H} + \sigma_{1}) > 0$ follows from \eqref{eq:a-H-sigma_2}.

Then the differential equation for $\vartheta$ is given by
\begin{equation}
  \label{eq:nonlinear_pendulum}
  \parentheses{ \od{\vartheta}{t} }^{2}
  = c^{2}\parentheses{ a^{2} + 4(H - \sigma_{1}^{2}) + 4a\sqrt{H} \cos\vartheta } \\
  = 2(E + A \cos\vartheta),
\end{equation}
where we set
\begin{equation*}
  E \defeq \frac{c^{2}}{2} \parentheses{ a^{2} + 4(H - \sigma_{1}^{2}) },
  \qquad
  A \defeq 2c^{2}a\sqrt{H}.
\end{equation*}
This is the differential equation for a nonlinear pendulum.

For example, if $A < E$, it corresponds to oscillatory solutions of the nonlinear pendulum, and one obtains the solution
\begin{equation*}
  \vartheta(t) = 2 \arcsin\parentheses{ \sn\parentheses{m, F(m,\vartheta_{0}/2) \pm \sqrt{\frac{E + A}{2}}\,t } },
\end{equation*}
where $m \defeq \sqrt{2A/(E + A)} < 1$, and $F$ is the elliptic integral of the first kind, i.e.,
\begin{equation*}
  F(m, \vartheta) \defeq \int_{0}^{\vartheta} \frac{1}{\sqrt{1 - m\sin^{2}\theta}}\,d\theta,
\end{equation*}
and $\sn$ is the Jacobi elliptic function, i.e.,
\begin{equation*}
  \sn^{-1}(m, x) \defeq \int_{0}^{x} \frac{1}{\sqrt{(1 - \xi^{2})(1 - m\,\xi^{2})}}\,d\xi
  = F(m, \sin^{-1}x),
\end{equation*}
that is, $\sn(m, F(m,\vartheta)) = \sin\vartheta$.
Then the angle $\theta$ is given in terms of $\vartheta$ as follows:
\begin{equation*}
  \theta(t) = 2 \arctan \parentheses{ \sqrt{ \frac{a - 2(\sqrt{H} - \sigma_{1})}{a + 2(\sqrt{H} + \sigma_{1})} } \tan(\vartheta(t)/2) }.
\end{equation*}
Therefore, we obtain $\phi_{1}$ and $\phi_{2}$ (and hence $\varphi_{1}$ and $\varphi_{2}$) by quadrature using \eqref{eq:phi-diffeq}; similarly we obtain the position $\mathbf{x}$ of the center of the sphere by quadrature using \eqref{eq:OptimalSystem} as well.

To summarize, we have the following:
\begin{theorem}
  \label{thm:integrability}
  The optimal control problem~\eqref{eq:OCP1} is completely integrable.
  Particularly, if the condition $a - 2(\sqrt{H} - \sigma_{1}) > 0$ is satisfied, then its solution $(\varphi(t), \mathbf{x}(t))$ is obtained by quadrature using a solution of the nonlinear pendulum equation~\eqref{eq:nonlinear_pendulum}.
\end{theorem}

\begin{remark}
  The above calculations and the result are reminiscent of a similar result by \citet{Ju1993} (see also \citet[Section~14.3]{Ju1997}) on the plate-ball system---rolling a ball on the plane by moving a plate attached at the top of the ball.
  Specifically, the result says that each extremal path of the center of the sphere---$\mathbf{x}(t) = (x_{1}(t), x_{2}(t))$ in our notation---connecting two translational \textit{and} rotational configurations is Euler's elastica.
  One difference is that our curve is in the shape space or the $\varphi_{1}$-$\varphi_{2}$ plane as opposed to the $x_{1}$-$x_{2}$ plane; another difference is that it is not exactly Euler's elastica.
  Given the differential equation~\eqref{eq:phi-diffeq}, our curve \textit{would be} Euler's elastica if $\theta(t)$ satisfied the nonlinear pendulum equation~\eqref{eq:nonlinear_pendulum}.
  In fact, this is the case with the plate-ball system with $(\varphi_{1},\varphi_{2})$ being replace by $(x_{1},x_{2})$.
  However, in our case, it is $\vartheta$---defined as a slight deformation of $\theta$ in \eqref{eq:vartheta}---that satisfies the nonlinear pendulum equation~\eqref{eq:nonlinear_pendulum}.
  As a result, our curve in the $\varphi_{1}$-$\varphi_{2}$ plane is a slight deformation of Euler's elastica; see Example~\ref{ex:OCP1} below.
\end{remark}

\begin{example}
  \label{ex:OCP1}
  We set the parameters as follows: $r = 1$, $\rho = 0.3$, $h = 0.75$, $w = 0.8$, $J/I_{\rm s} = 5$, and $T = 10$.
  Consider the problem of maneuvering the center of spherical rolling robot from the origin to $(1,1)$ on the $x_{1}$-$x_{2}$ plane after each wheel makes 5 revolutions in displacement (not necessarily the total revolutions), i.e., $\mathbf{x}(0) = (0,0)$, $\varphi(0) = (0, 0)$, $\mathbf{x}(T) = (1,1)$, and $\varphi(T) = (10\pi, 10\pi)$.
  This turns out to be the case with $A < E$ discussed above.
  
  Figure~\ref{fig:OCP1} shows the optimal trajectory of the wheels in the (covering space of) the shape space $S$ and of the center of the sphere in $\R^{2}$.
  \begin{figure}[htbp]
    \centering
    \subfigure[Trajectory of the wheels in the covering space $\R \times \R$ of the shape space $S = \mathbb{S}^{1} \times \mathbb{S}^{1}$.]{
      \includegraphics[width=.465\linewidth]{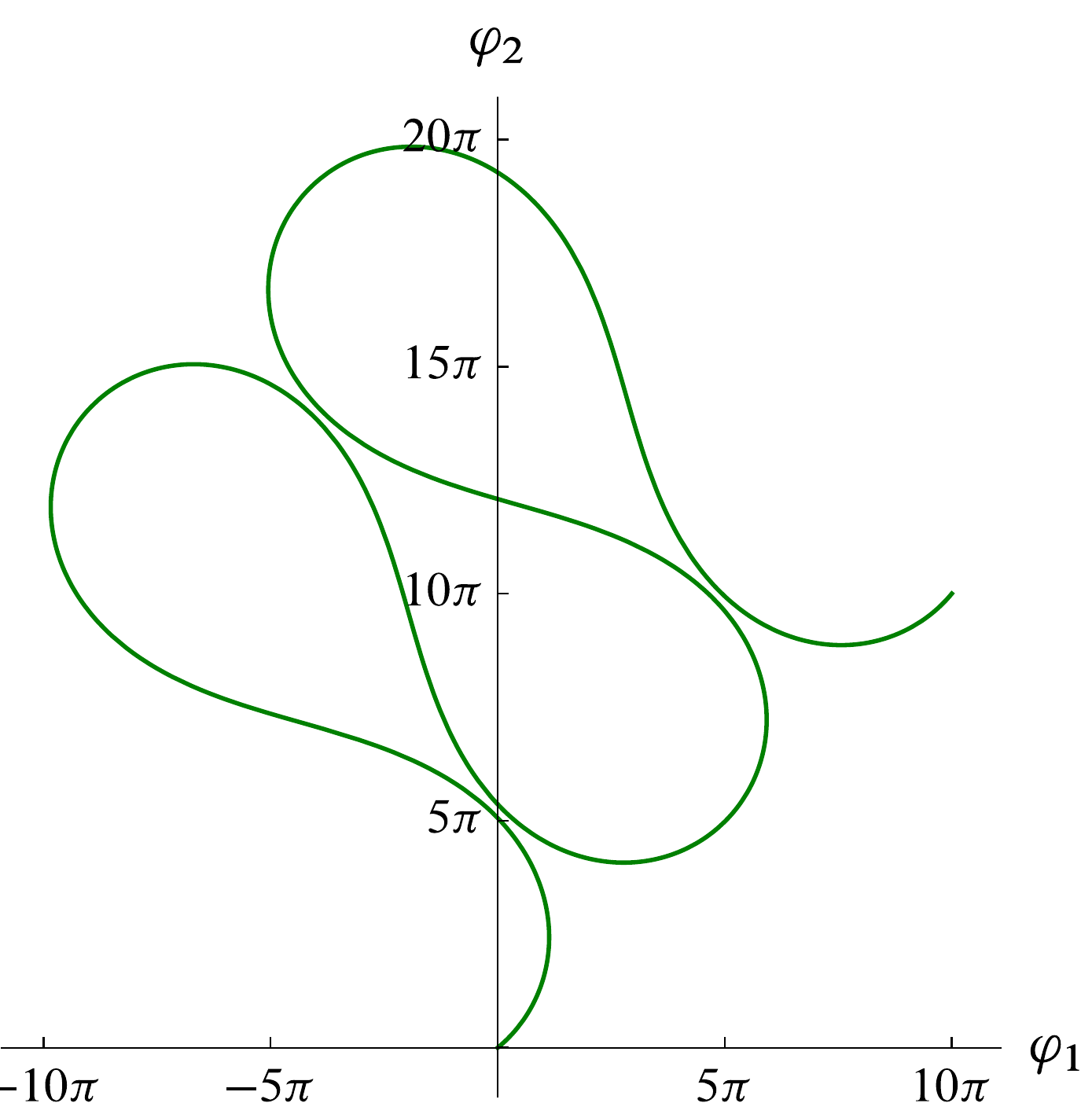}
    }
    \quad
    \subfigure[Trajectory of the center of the sphere in $\R^{2}$]{
      \includegraphics[width=.465\linewidth]{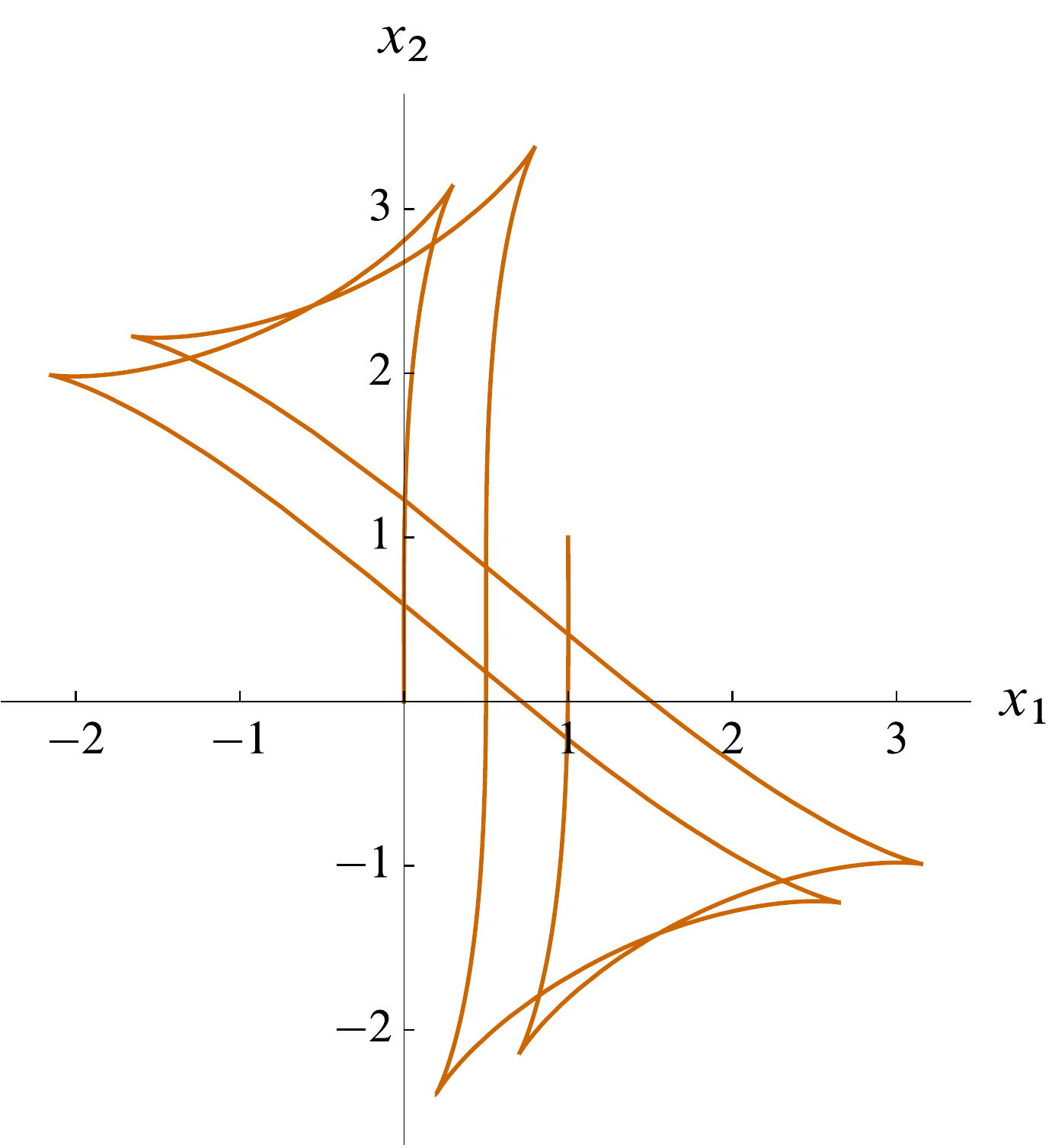}
    }
    \captionsetup{width=0.9\textwidth}
    \caption{
      Solution of the optimal control problem~\eqref{eq:OCP1} for Example~\ref{ex:OCP1}---trajectories of the wheels and of the center of the sphere.
      One can see that the trajectory of the wheels in the $\varphi_{1}$-$\varphi_{2}$ plane is very similar to one of Euler's elasticas.
      On the other hand, the trajectory of the center of the sphere is much more complicated and makes several switches in its direction.
    }
    \label{fig:OCP1}
  \end{figure}
\end{example}

\section*{Acknowledgments}
I would like to thank Vakhtang Putkaradze for his helpful comments and discussions.
This work was partially supported by NSF grant CMMI-1824798.

\bibliography{Sphero}
\bibliographystyle{plainnat}

\end{document}